\documentclass{amsart}

\def \mytitle {Convergence rate of numerical scheme for SDEs with a distributional drift in Besov space}
\def \myauthor {Luis Mario Chaparro J\'aquez, Elena Issoglio, Jan Palczewski}

\usepackage{amsmath}
\usepackage{amsfonts}
\usepackage{amsthm}
\usepackage[foot]{amsaddr}
\usepackage[utf8]{inputenc}
\usepackage[backend=biber]{biblatex}
\usepackage{bbm}
\usepackage{a4wide}
\usepackage[most, many, breakable]{tcolorbox}
\usepackage{hyperref}
\usepackage{cancel}
\usepackage{xcolor}
\usepackage{soul} 
\usepackage[linesnumbered,lined,boxed]{algorithm2e}
\usepackage{pgf}
\usepackage{lmodern}
\usepackage{import}

\graphicspath{ {./Figures/} }

\newtheorem{proposition}{Proposition}[section]
\newtheorem{theorem}[proposition]{Theorem}
\newtheorem{corollary}[proposition]{Corollary}
\newtheorem{lemma}[proposition]{Lemma}

\newtheorem{assumption}{Assumption}
\newtheorem{remark}[proposition]{Remark}

\definecolor{leedsGreen}{RGB}{0, 80, 47}
\definecolor{leedsRed}{RGB}{196, 18, 48}
\definecolor{leedsBlack}{RGB}{0, 0, 0}
\definecolor{leedsCream}{RGB}{246, 241, 228}

\newcommand{\holder}[1]{$#1$\nobreakdash-H\"older} 
\newcommand{\ctc}[1]{C_{T} \mathcal{C}^{#1}} 
\newcommand{\hz}[1]{\mathcal{C}^{#1}} 
\newcommand{\R}{\mathbb{R}} 
\newcommand{\N}{\mathbb{N}} 
\newcommand{\expectation}{\mathbb{E}} 
\newcommand{\suppDisc}{\mathfrak{M}}

\DeclareMathOperator\sgn{sgn}

\allowdisplaybreaks

\AtEveryBibitem{
\ifentrytype{article}
   {\clearfield{doi}
   \clearfield{url}}
   {}
\ifentrytype{unpublished}
   {\clearfield{doi}\clearfield{url}}
   {}
\ifentrytype{book}
   {\clearfield{doi}\clearfield{url}}
   {}
\clearfield{note}
\clearfield{urldate}
\clearfield{pubstate}
\clearfield{pagetotal}
\clearfield{location}
\clearfield{eprintclass}
\clearfield{doi}
\clearfield{issn}
\clearfield{isbn}
}

\hypersetup{
	pdftitle={\mytitle}, 
	pdfsubject={Probability}, 
	pdfauthor={\myauthor},
	pdfkeywords={SDEs, Besov Space},
  colorlinks,
  linkcolor={red!80!black},
  citecolor={blue!80!black},
  urlcolor={green!80!black}
}
\title[Convergence rate of numerical scheme for SDEs with a distributional drift]{\mytitle}
\author{Luis Mario Chaparro J\'aquez$^{\dag1}$}
\author{Elena Issoglio$^{\ddag2}$}
\author{Jan Palczewski$^{\dag3}$}
\address{$^\dag$School of Mathematics, University of Leeds, Leeds, LS2 9JT, United Kingdom}
\address{$^\ddag$Dept. of Mathematics ``G. Peano'', University of Turin, Via Carlo Alberto 10, 10123, Torino, Italy}
\email{$^1$mmlmcj@leeds.ac.uk}
\email{$^2$elena.issoglio@unito.it}
\email{$^3$j.palczewski@leeds.ac.uk}

\date{\today}
\begin{document}

\begin{abstract}
This paper is concerned with numerical  solutions of one-dimensional SDEs with the drift being a generalised function {in the spatial variable}, in particular {being a $\frac12$-H\"older continuous function of time taking values in a} H\"older-Zygmund space 
$\hz{-\gamma}$ of negative order $-\gamma<0$. We design an Euler-Maruyama numerical scheme and prove its convergence, obtaining  an upper bound for the  strong $L^1$ convergence rate. We finally implement the scheme and discuss the results obtained.


\end{abstract}

\maketitle


\section{Introduction} \label{sec:intro}

This paper is concerned with the Euler-Maruyama scheme and its rate of convergence for a one-dimensional stochastic differential equation (SDE) of the form
\begin{equation}
    \label{eq:sde_S}
    X_{t}
	=
	x
	+
	\int_{0}^{t}
	b(s, X_{s})
	ds
	+
	W_{t},
\end{equation}
where
$ (W_{t} )_{t\ge0}$
is a Brownian motion and the drift 
$ b (t, \cdot)$
 belongs to the space of Schwartz distributions
$ \mathcal{S}'(\R) $ for all $t\in [0,T]$.
More precisely, the map
$ t \mapsto b(t) $
is
$\frac12$-H\"older 
continuous for any
$ t \in [0, T] $ with values in the H\"older-Zygmund space 
$\hz{-\gamma}$ of negative order $-\gamma<0$, 
which we denote    by 
$ b \in  C_T^{1/2} \hz{-\gamma} $. For a precise definition of these spaces see Section \ref{subsec:function_spaces} below. 

SDEs with distributional coefficients have been studied by several authors in different settings and with different noises, starting from the early 2000s with \cite{bass_chen, flandoliSDEsDistributionalDrift2003, flandoliSDEsDistributionalDrift2004} and then in recent years by 
\cite{flandoliMultidimensionalStochasticDifferential2017, delarueRoughPaths1d2016,  cannizzaro, chaudru_menozzi, issoglioSDEsSingularCoefficients2023}.
In all these works the drift is a distribution and the authors investigate theoretical questions of existence and uniqueness of solution, without exploring  numerical aspects. The specific setting we consider here is the one studied in 
\cite{issoglioSDEsSingularCoefficients2023},
where the authors formulate the notion of solution  to \eqref{eq:sde_S} as a suitable  martingale problem (c.f. 
Section \ref{subsec:soln_sde}). SDE \eqref{eq:sde_S} is solved in  
\cite{issoglioSDEsSingularCoefficients2023}
in any general dimension $d$, 
and the notion of solution is intrinsically a weak solution, since it is formulated as a martingale problem. Here
we restrict ourselves to dimension $1$ {in which case there is a  strong solution $X$ (see Remark \ref{rm:strong}); the strong convergence studied in this paper can only be defined for strong solutions. The research of weak convergence for multi-dimensional SDEs of the above type is left for the future.}

The first results on numerical schemes for SDEs date back to the 1980s, see the book by Kloeden and Platen \cite{kloeden_platen} for the case of smooth coefficients.
On the other hand, numerical schemes for SDEs with low-regularity coefficients is an active area of research, but almost all contributions deal with SDEs with coefficients that are at least functions. {We refer to the introduction of \cite{kohatsu-higa_et.al} and of \cite{dareiotisQuantifyingConvergenceTheorem2021}  for a list of other relevant papers  and  a short summary of techniques used for numerical schemes with low-regularity coefficients.}
{In \cite{neunekirch_szolgyenyi} the authors study SDEs with additive noise in dimension 1 and with the drift in Sobolev spaces $W^{\gamma}_2\cap L^1 \cap L^\infty $ for $\gamma \in (0,1)$ and prove a strong rate of convergence of $\frac{1+\gamma}2 \wedge \frac34$ for the Euler-Maruyama scheme. Notice that the  drift has a positive Sobolev regularity of $\gamma$, hence it is a possibly discontinuous function.}
Paper \cite{dareiotisQuantifyingConvergenceTheorem2021} {extends this result to any dimension $d$ and removes the bound $\frac34$.  Indeed, they prove  a strong $L^p$-rate  of $\frac{1+\gamma}2$ when the drift $b$ is time-homogeneous  and an element of the homogeneous Sobolev space $\dot W_{d\vee2}^\gamma$ for $\gamma \in (0,1)$ and the diffusion coefficient is the identity. In \cite{dareiotisQuantifyingConvergenceTheorem2021} they   also consider multiplicative noise and prove a strong $L^p$-rate  of $\frac12$ for non-unitary $C^2$ diffusion coefficient  with a bounded measurable drift.} 
Another relevant paper is \cite{jourdainConvergenceRateEulerMaruyama2021}, where the authors deal with the case of $L^q$-$L^p$ drift and unitary diffusion. They prove that the weak error of the Euler-Maruyama scheme  is $\frac\alpha2$, where $\alpha$ is the distance from the singularity $\alpha:= 1- \left(\frac dp + \frac2q\right)$, {where $d$ is the dimension of the problem}. They also conjecture that their methods of proof should produce a rate of $\frac{1+\gamma}2$ for a time-inhomogeneous drift that belong to $C_T^{\frac\gamma2}C^\gamma$.
A different line of research investigates discontinuous drifts and possibly degenerate diffusion coefficients, where the discontinuities lie on finitely many points or hypersurfaces, see 
{\cite{leobacherStrongOrder12017, Leobacher_szolgyenyi, neuenkirchAdaptiveEulerMaruyamaScheme2019, Przybylowicz_szolgyenyi}
for more details, and \cite{szolgyenyi} for a review.}

Only a few works deal with numerical schemes for SDEs with distributional coefficients. In \cite{deangelisNumericalSchemeStochastic2022}, the SDE is like \eqref{eq:sde_S} but the drift $b$ belongs to a different distribution space, namely to the fractional Sobolev space of negative order  $b \in C^{\kappa}_T H^{-\gamma}_{{(1-\gamma)}^{-1}, q}$ for $\kappa \in (\frac{1}{2}, 1)$, $\gamma \in (0, \frac{1}{4})$ and $q \in (4, \frac{1}\gamma)$. The authors obtain a strong $L^1$-rate of convergence depending on $\gamma$ which vanishes as $-\gamma$ approaches the boundary $-\frac14$, and tends to  $\frac16$ for when $-\gamma$ approaches $0$ (i.e., when it approaches measurable drifts). 
In \cite{goudenegeNumericalApproximationSDEs2022}, the authors study SDEs in $d$-dimensions with drifts in negative Besov spaces $B^{-\gamma}_p$, and the noise is a fractional Brownian motion with Hurst index $H\in (0,\frac{1}{2})$. They require that $1-\frac{1}{2H} < -\gamma-\frac{d}{p} <0$, i.e., the roughness of the drift is compensated by the roughness of the noise. The case $p=\infty$, $d=1$ and $H=\frac{1}{2}$  would correspond to our case, but this combination of parameters violates the above condition as the left-hand side becomes $0$. Techniques used in \cite{goudenegeNumericalApproximationSDEs2022} cannot be easily extended to our case. Indeed, for their proofs of convergence the authors rely on the well-known fact that a rougher noise gives more regularity to the solution, hence allowing for a rougher drift coefficient $b$ (or a higher dimension).

In this paper, we set up a two-step numerical scheme.
The first step is to regularise the distributional drift with the action of the heat semigroup, which gives a smooth function and allows us to use Schauder estimates to control the approximation error bounds for the solution of the SDE with smoothed drift. Proving this step is the bulk of the paper.
{The second step is to bound the error of the Euler-Maruyama scheme}, which requires ad-hoc estimates (rather than {standard EM} estimates that can be found in most of the literature) to be able to control the constant in front of the rate in terms of the properties of the smoothed drift. To do so, we borrow ideas and results from \cite{deangelisNumericalSchemeStochastic2022}, but we still have to prove a  delicate $L^1$-bound of the local time of the error process (see Lemma \ref{lemma:local-time-at-0}). Notice that we consider the $L^1$ strong error, and not the more common $L^2$ error, because the diffusion coefficient of an auxiliary process $(Y_t)$ which is used to define a virtual solution to SDE \eqref{eq:sde_S} is only H\"older continuous; hence, the quadratic variation of the square of the difference between the solution and its approximation is of a lower order than the square of the approximation error, so a Gronwall lemma argument cannot be applied, see Remark \ref{rem:L1} for more details. Finally, we link the smoothing parameter in Step 1 and the time step in Step 2 to obtain a one-step scheme and its convergence rate.

This paper is organised as follows. In Section \ref{sec:preliminaries} we set up the problem and the notation, recalling all relevant theoretical tools such as H\"older-Zygmund spaces, the heat kernel and semigroup, Schauder estimates and the notion of virtual solution to SDE \eqref{eq:sde_S}. In Section \ref{sec:numerics_results} we describe the numerical scheme and state the main result (Theorem \ref{th:convergence_rate_es}) which provides a convergence rate in terms of the regularity parameter $\gamma$ (Corollary \ref{cor:rate}). Section  \ref{sec:proof} contains the proof of the building bloc of the main theoretical result (Proposition \ref{prop:reg_to_original}), which is a bound on the difference between the solution to the original SDE \eqref{eq:sde_S} and its approximation after smoothing the drift $b$. Here is where we make use of the bound on the local time borrowed from \cite{deangelisNumericalSchemeStochastic2022}. Finally in Section \ref{sec:numerics}, we describe a numerical implementation of the scheme and analyse numerical results. It is striking to see that the empirical convergence rate seems to be $\frac12 - \frac{\gamma}2$, which would be the extension of the rate found in \cite{dareiotisQuantifyingConvergenceTheorem2021} if they allowed for negative regularity index (and hence for distributional drifts). A straightforward application of their techniques is not possible and further investigations in this direction, for example {using the stochastic sewing lemma} introduced in \cite{leStochasticSewingLemma2020}, are left for future research.

\section{Preliminaries}\label{sec:preliminaries}

\subsection{Notation}\label{subsec:function_spaces}
For a function $f: [0, T] \times \R \to \R$ that is sufficiently smooth, we denote by $f_t$ the partial derivative with respect to $t$, by $f_x$ the partial derivative with respect to $x$, and by $f_{xx}$ the second partial derivative with respect to $x$. For a function $g:\R \to \R$ sufficiently smooth we denote its derivative by $g'$.

We now recall some useful definitions and facts from the literature. First of all, let
	$ \mathcal{S}(\R) $
	be the space of Schwartz functions on $ \R $
	and 	$ \mathcal{S}'(\R) $
	the space of tempered distributions.
	We denote $ (\cdot)^{\wedge} $
	and	$ (\cdot)^{\vee} $
	the Fourier transform and inverse Fourier transform on
	$ \mathcal{S} $
	respectively, extended to 
	$ \mathcal{S}' $
	in the standard way.
	For $\gamma \in \R$, the H\"older-Zygmund space is  defined by
	\begin{equation}
	    \label{eq:hs_space}
		\hz{\gamma}(\R)
		=
		\Big\{ 
			f \in \mathcal{S}'
			:
			\left\| f \right\|_{\gamma}
			:=
			\sup_{j \geq -1}
			2^{j \gamma}
			\Big\| \big( \phi_{j} \hat{f}\, \big)^{\!\vee} \Big\|_{L^{\infty}}
			<
			\infty
		\Big\},
	\end{equation}
where 	$( \phi_{j})_j$ is any partition of unity. The H\"older-Zygmund space $ \hz{\gamma}(\R)$ is also known as Besov space $B^\gamma_{\infty, \infty}(\R)$. For more details see 
\cite{triebelTheoryFunctionSpaces2010, bahouriFourierAnalysisNonlinear2011}.
To shorten notation we write
$ \hz{\gamma} $
instead of $ \hz{\gamma} (\R) $.
Note that if
$ \gamma \in \R^{+} \setminus \N $
the space above coincides with the classical H\"older space. These spaces will be used widely in the paper, so we recall the norms that we will use in the paper. If
$ \gamma \in (0,1) $,
the classical 
\holder{\gamma}
norm 
\begin{equation}
    \label{eq:equiv_norm01}
    \left\| f \right\|_{L^{\infty}}
	+
	\sup_{
	\substack{
		x \neq y
		\\
		\left| x - y \right| < 1
	}
	}
	\frac{\left| f(x) - f(y) \right|}{\left| x - y \right|^{\gamma}}
\end{equation}
is an equivalent norm in $ \hz{\gamma} $.
If
$ \gamma \in (1,2) $
an equivalent norm is 
\begin{equation}
    \label{eq:equiv_norm12}
    \left\| f \right\|_{L^{\infty}}
	+
    \left\| f' \right\|_{L^{\infty}}
	+
	\sup_{
	\substack{
		x \neq y
		\\
		\left| x - y \right| < 1
	}
	}
	\frac{\left|  f'(x) -  f'(y) \right|}{\left| x - y \right|^{\gamma-1}}.
\end{equation}
We will write $\ctc{\gamma} := C([0, T]; \hz{\gamma})$ with the norm 
\[
\|f\|_{\ctc{\gamma}}:= \sup_{t\in[0,T]} \|f(t)\|_{\hz{\gamma}}.
\]
We will also use a family of equivalent norms $\|\cdot \|^{(\rho)}_{\ctc{\gamma}}$, for $\rho\geq0$, given by
\[
\|f\|^{(\rho)}_{\ctc{\gamma}}  := \sup_{t\in[0,T]} e^{-\rho(T-t)} \|f(t)\|_{\hz{\gamma}}.
\]
Indeed, it is easy to see that 
\begin{equation}\label{eq:rho_eq_norm}
    \|f\|_{\ctc{\gamma}} \leq e^{\rho T}  \|f\|_{\ctc{\gamma}}^{(\rho)}. 
\end{equation}
For any given
	$ \gamma \in \R $
	we denote by
	$ \hz{\gamma+} $
	and
	$ \hz{\gamma-} $
	the following spaces
	\begin{equation}
	    \label{eq:c+}
	    \hz{\gamma+} 
		:= 
		\cup_{\alpha>\gamma}
		\hz{\alpha},
	\end{equation}
	\begin{equation}
	    \label{eq:c-}
	    \hz{\gamma-} 
		:= 
		\cap_{\alpha<\gamma}
		\hz{\alpha}.
	\end{equation}
Similarly, we also write 
$ \ctc{\gamma+} := \cup_{\alpha > \gamma} C_T \mathcal C^\alpha$.

The following bound in H\"older-Zygmund spaces will be useful later, and it is known as the Bernstein inequality. 
\begin{lemma}[Bernstein inequality]\label{lemma:bernstein_inequality}
    For any $\gamma \in \R$, there is $c > 0$ such that 
    \begin{equation}
        \label{eq:bernstein_inequality}
        \left\| f' \right\|_{\gamma}
        \leq
        c
        \left\| f \right\|_{\gamma + 1}, \qquad f \in \hz{\gamma + 1}.
    \end{equation}
\end{lemma}

Let $\kappa \in (0,1)$. We denote by $C_T^\kappa L^\infty$ the space of $\kappa$-H\"older continuous functions with values in $L^\infty$  (the space of bounded functions),
and by $C^\kappa_T C^1_b$ the space of  $\kappa$-H\"older continuous functions from $[0,T]$ with values in the space of $C^1$ functions which are bounded and have a bounded derivative. We define the following norms and seminorms for $g \in  C_T^\kappa L^\infty$:
	\begin{equation}\label{eq:norm_infty}
		\left\| g \right\|_{\infty, L^\infty}
		:=
		\sup_{t\in[0,T]} \sup_{x\in\R}\left | g(t,x) \right|
	\end{equation}
	and
	\begin{equation}\label{eq:norm_infinity_holder}
		[g]_{\kappa, L^{\infty}} := \sup_{t,s \in [0,T], t \neq s} \frac{\|g(t) - g(s)\|_{L^{\infty}}}{|t - s|^{\kappa}}.
	\end{equation}
Notice that if $g\in C_T^\kappa C^1_b$ then $g_x \in C_T^\kappa L^\infty$.

We finish this section by introducing an asymptotic relation between functions. For functions $ f, g $ defined on an unbounded subset of $ \R^{+} $, we write $f(x) = o (g(x))$ if $\lim_{x \to \infty} |f(x) | / |g(x)| = 0$.

\subsection{Standing assumption}\label{subsec:assumptions}
The following assumption will hold  throughout the paper.
\begin{assumption}\label{ass:b_betahat}
We have $ b \in  C_T^{1/2} \hz{(-\hat \beta)+}$ for some $ 0 < \hat{\beta} < 1/2 $.
\end{assumption}

We will derive our numerical scheme and prove its convergence under the above assumption, but solutions to the SDE \eqref{eq:sde_S} exist under a weaker assumption $ b \in  C_T \hz{- \hat \beta} $ {as explained in detail in Section \ref{subsec:soln_sde} below,} and this fact will be exploited in the derivation of the rate of convergence. Although solutions to SDE \eqref{eq:sde_S} exist in higher dimensions, we work in dimension $1$ because in this case one can construct a  strong solution, see Remark \ref{rm:strong}, which is fundamental to the definition of strong convergence error.

\subsection{Heat kernel and heat semigroup}\label{subsec:heat_kernel}
We will use heat kernel smoothing to derive a sequence of approximating  SDEs. Here we introduce notation and provide background information about the action of the heat semigroup on elements of $\hz{\gamma}$.

The  function 
  	\begin{equation}
  	    \label{eq:heat_kernel}
        p_{t}(x) = {\frac{1}{\sqrt{2 \pi t}} e^{-\frac{x^{2}}{2t}}}
  	\end{equation}
is called the heat kernel and it is  the fundamental solution to the heat equation.
  	The operator acting as a convolution of the heat kernel with a (generalised) function is called heat semigroup, and it is denoted by 
  	$ P_{t} $: for any 	
  	$ g \in \mathcal{S} $ we have 
  	\begin{equation}
  	    \label{eq:heat_semigroup}
  	    \left(P_{t} g\right) (y) = \left(p_{t} \ast g\right) (y) = \int_{\R} p_{t}(x) g(y - x) dx.
  	\end{equation}
The semigroup $P_t$ can be extended to $\mathcal S'$ by duality.   For a distribution $ \varrho \in \mathcal{S}' $,
the convolution and derivative commute as mentioned in
\cite[Section 5.3]{renardyIntroductionPartialDifferential2004}
and
\cite[Remark 2.5]{issoglioPdeDriftNegative2022},
that is, 
\begin{equation}
    \label{eq:comm_sem_der}
(p_{t} \ast \varrho)' 
	=  p'_{t} \ast  \varrho
	= p_{t} \ast  \varrho'.
\end{equation}
This fact is useful for efficient construction of regularised SDEs when $b = B_x$ for some function $B \in C^{1/2}_T \hz{(1-\hat \beta)+}$ as we do in the numerical example studied in Section \ref{sec:numerics}.

We recall the so-called Schauder estimates which quantify the effect of heat semigroup smoothing.
\begin{lemma}[Schauder estimates]
    \label{lemma:schauder_estimates}
    For any $\gamma \in \R$, {and $\theta \geq 0$ there is a positive constant $c$} such that for any   $f \in \hz{\gamma}$
    \begin{equation}
        \label{eq:se_Pt}
        \left\| P_t f \right\|_{\gamma + 2 \theta} \leq c t^{-\theta} \left\| f \right\|_{\gamma}.
    \end{equation}
   {Also, for any $\gamma \in \R$, and $0 < \theta < 1$ there is a positive constant $c$} such that for any   $f \in \hz{\gamma + 2\theta}$ then
   \begin{equation}
       \label{eq:se_Pt-I}
       \left\| P_t f - f \right\|_{\gamma} \leq c t^{\theta} \left\| f \right\|_{\gamma + 2 \theta}.
   \end{equation}
\end{lemma}
For a proof of the result above see \cite[Theorem 2.10]{issoglioDegenerateMcKeanVlasovEquations2024} with $d = 0$ and $N = 1$.

\subsection{Existence of solutions to the SDE}\label{subsec:soln_sde}

{From \cite{issoglioPdeDriftNegative2022, issoglioSDEsSingularCoefficients2023} we recall the main results on construction, existence and uniqueness of solutions to SDE \eqref{eq:sde_S} under the following assumption:}

\begin{equation}\label{eqn:ass_b}
 b \in C_T \hz{(- \beta)+}, \qquad \beta \in (0, 1/2).
\end{equation}
{There exist} two equivalent notions of solution to SDE \eqref{eq:sde_S}: virtual solutions and solutions via martingale problem. The formulation via  martingale problem is omitted and the reader is referred to \cite{issoglioSDEsSingularCoefficients2023}. We describe below the formulation via virtual solutions, because it is particularly suited to our analysis of the numerical scheme.

For $\lambda > 0$, let us consider a Kolmogorov-type PDE
\begin{equation}
		\label{eq:kolmogorov}
	  	\begin{cases}
			u_t + \frac{1}{2} u_{xx} + b u_x = \lambda u - b
	  	  	\\
	  	  	u(T) = 0
	  	\end{cases}
	\end{equation}
with the solution understood in a mild sense, i.e., as a solution to the integral equation
	\begin{equation}
		\label{eq:mild_solution}
		u(t) 
		=
		\int_{t}^{T} P_{s - t} \left(  u_x (s) b(s) \right) ds - \int_{t}^{T} P_{s-t} \left( \lambda u(s) - b(s) \right) ds, \quad \forall
	 t \in [0, T] .
	\end{equation}
It is shown in \cite[Theorem 4.7]{issoglioPdeDriftNegative2022} that a solution $u$ exists in $\ctc{(2 - \beta)-} $ and is unique in $C_T \hz{(1+\beta)+}$. Hence, $u_x$ is  \holder{\alpha} continuous for any $\alpha <1-\beta$.	By \cite[Proposition 4.13]{issoglioPdeDriftNegative2022}, we have $\|u_x\|_\infty<1/2$ for $\lambda$ large enough. 

From now on we fix  $\lambda$ large enough. 
By \cite[Proposition 4.13]{issoglioPdeDriftNegative2022}, the mapping 
\begin{equation}\label{eq:phi}
    \phi(t, x):= x + u(t,x)
\end{equation} 
is invertible in the space dimension, and we  denote this space-inverse by $\psi(t, \cdot)$. Consider now a  weak solution $Y$ to SDE
\begin{equation}
	\label{eq:Y_def}
	Y_t = y_0 + \lambda \int_0^t u(s, \psi(s, Y_s)) ds + \int_0^t (  u_x (s, \psi(s, Y_s)) + 1 ) dW_s.
\end{equation}
{We say that $X_t := \psi (t, Y_t)$} is a \emph{virtual solution}\footnote{We borrow here the term virtual solution from \cite{flandoliMultidimensionalStochasticDifferential2017}, where the authors use an analogous equation for $Y$ to define the solution of an SDE with distributional drift $b$ in a fractional Sobolev space. In \cite{issoglioSDEsSingularCoefficients2023} the authors instead define the solution via martingale problem, but the equivalence with the notion of virtual solutions follows from their Theorem 3.9.} to SDE \eqref{eq:sde_S}. Clearly, $Y_t = \phi (t, X_t) = X_t + u(t, X_t)$ solves \eqref{eq:Y_def}
when $X$ is a  weak solution to the equation
\begin{equation}
	    \label{eq:virtual_int_eq}
	    X_{t} = x + u(0, x) - u(t, X_{t}) + \lambda \int_{0}^{t} u(s, X_{s}) ds + \int_{0}^{t} \left( u_x (s, X_{s}) + 1 \right) dW_{s},
	\end{equation}
so any solution $(X_t)$ of \eqref{eq:virtual_int_eq} is also a virtual solution to SDE \eqref{eq:sde_S}. We also note that the virtual solution does not depend on $\lambda$ thanks to the links between virtual solutions and martingale problem developed in \cite{issoglioSDEsSingularCoefficients2023}; the reader is refered to the aforementioned paper for a complete presentation of those links.

We complete this section by showing that the virtual solution of SDE \eqref{eq:sde_S} is a strong solution, in the sense that the solution $Y$ to SDE \eqref{eq:Y_def} is a  strong solution.

\begin{lemma}\label{rm:strong}
Under condition \eqref{eqn:ass_b}, there exists process $(Y_t)$ on the original probability space with the Brownian motion $(W_t)$ that satisfies \eqref{eq:Y_def}, and this process is unique up to modifications. Hence, there is a strong solution $(X_t)$ of SDE \eqref{eq:sde_S}.
\end{lemma}
\begin{proof}
The drift is Lipschitz in $y$ uniformly in $t$ and the volatility is \holder{\alpha} in $y$ uniformly in $t$, for some $\alpha>\frac12$. This allows us to use arguments in the proof of \cite[Ch.~IV, Thm.~3.2]{ikeda} (see also the corollary following it) to claim the existence and uniqueness of a strong solution $Y$ to SDE \eqref{eq:Y_def}.
\end{proof}


\section{The  numerical scheme and main results}\label{sec:numerics_results}

{Our numerical scheme for SDE} \eqref{eq:sde_S} is based on two approximations. The first one replaces the distributional drift with a sequence of functional drifts so that {the solutions of the respective SDEs converge} to the solution of the original SDE. Subsequently, the approximating SDEs are simulated with an Euler-Maruyama scheme. We will then balance errors coming from the approximation of the drift and the discretisation of time to maximise the rate of convergence.

We regularise the drift {by applying the heat semigroup}. For a real number $N > 0$, we define
\begin{equation}\label{eq:bN}
b^N = P_{\frac1N} b.
\end{equation}
{By an application of Lemma \ref{lemma:schauder_estimates} with parameter $t = 1/N$ we know that $b^N(t, \cdot) \in  \hz{\gamma}$ for any $\gamma>0$ and for all $t\in[0,T]$, thus  we  also} have $b^N (t, \cdot) \in C^1_b$, $t\in[0,T]$, hence, it is Lipschitz continuous in the spatial variable.
Hence, the SDE
\begin{equation}\label{eqn:X^N}
dX^N_t = b^N(t, X^N_t) dt + dW_t
\end{equation}
has a unique strong solution.
{Now we write the standard EM scheme for the above SDE.} For 
$m \in \mathbb{N}$
we take an equally spaced partition 
$t_k = t^m_{k} = kT/m$, 
$k = 0, \ldots, m$,
of the interval 
$[0, T]$.
We define
	\begin{equation}
	    \label{eq:time_sup}
	    k(t) = k^m(t) = \max \left\{ k: t_{k} \leq t \right\}, \qquad t \in [0, T].
	\end{equation}
Consider an Euler-Maruyama approximation of $(X^N)$ with $m$ time steps
	\begin{equation}
	    \label{eq:m_em_approx}
	    X^{Nm}_{t} = x + \int_{0}^{t} {b^N}\big(t_{k(s)}, X^{Nm}_{t_{k(s)}}\big) ds + W_{t}, \qquad t \in [0, T].
	\end{equation}
We first obtain a bound on the strong error between the approximation $(X^{Nm})$ and the process $(X^N)$ with explicit dependence of constants on the properties of the drift $b^N$. This is needed in order to balance the smoothing via choice of $N$ and the number of time steps $m$ to optimise the convergence rate of the Euler-Maruyama approximation to the true solution $(X)$. Following arguments in the proof of \cite[Prop.~3.4]{deangelisNumericalSchemeStochastic2022} we obtain the following result.

\begin{proposition}\label{pr:Euler_conv}
 Assume that $b^N \in C^{1/2}_T L^\infty \cap L^\infty_T C_b^1$. Then
\[
\sup_{0 \leq t \leq T}
		\expectation \left[\left| X^{Nm}_{t} - X^N_{t}\right|\right]  \leq  A^N m^{-1} + B^N m^{-1/2},
\]
where 
\begin{align*}
A^N &=\left\|b^N\right\|_{\infty,L^{\infty}}
		 	\left( 1 + \left\| b^N_x\right\|_{\infty,L^{\infty}} \right),\\
B^N &= 	\left\|  b^N_x \right\|_{\infty, L^{\infty}} + \left[b^N\right]_{\frac12, L^{\infty}}.
\end{align*}
\end{proposition}

\begin{corollary}\label{cor:Euler_conv}
Under Assumption \ref{ass:b_betahat}, the condition of Proposition \ref{pr:Euler_conv} is satisfied and for any $\epsilon > 0$ there is a constant $c>0$ such that 
$$A^N \le c N^{\frac{\epsilon + \hat{\beta}}{2}} \left\| b \right\|_{\ctc{-\hat{\beta}}} \left( 1+ N^{\frac{\epsilon + \hat{\beta} + 1}{2}} \left\| b \right\|_{\ctc{-\hat{\beta}}} \right)$$ 
and 
$$B^N \le c  N^{\frac{\epsilon + \hat{\beta} + 1}{2}} \left\| b \right\|_{\ctc{-\hat{\beta}}} + c N^{\frac{\epsilon+\hat\beta}2} \|b\|_{C_T^\frac12 \hz{-\hat\beta}} .$$
\end{corollary}
\begin{proof}
We first show that  $b^N \in C^\frac12 _T L^\infty \cap L^\infty_T C^1_b$. In the proof, a constant $c$ may change from line to line.

We apply Lemma \ref{lemma:schauder_estimates} with $\epsilon$ from the statement of the corollary, $\theta=\frac{\epsilon+\hat\beta}{2}$ and $\gamma= -\hat\beta$ to get 
\begin{align*}
\|b^N(t,\cdot) - b^N(s, \cdot)\|_{L^\infty} 
&=  
\|P_{\frac1N} (b(t,\cdot) - b(s, \cdot))\|_{L^\infty}
\leq c \|P_{\frac1N} (b(t,\cdot) - b(s, \cdot))\|_{\hz{\epsilon}}\\
&\leq c N^{\frac{\epsilon+\hat\beta}2} \|b(t,\cdot) - b(s, \cdot)\|_{\hz{-\hat\beta}},
\end{align*}
where the first inequality is by \eqref{eq:equiv_norm01}. Hence,
\begin{equation}\label{eq:er04}    
[b^N]_{\frac12, L^\infty} \leq c N^{\frac{\epsilon+\hat\beta}2} \|b\|_{C_T^\frac12 \hz{-\hat\beta}}.
\end{equation}
By the same arguments applied to $b^N(t, \cdot)$, we obtain $\|b^N(t, \cdot) \|_{L^\infty} \le  c N^{\frac{\epsilon+\hat\beta}2} \|b(t,\cdot)\|_{\hz{-\hat\beta}}$, which yields
\begin{equation}
		\label{eq:er02}
		\left\| b^{N} \right\|_{\infty,L^{\infty}}
		\leq
		\left\| b^{N} \right\|_{\ctc{\epsilon}}
		\leq
		c N^{\frac{\epsilon + \hat{\beta}}{2}} \left\| b \right\|_{\ctc{-\hat{\beta}}}.
\end{equation}
Bounds \eqref{eq:er04} and \eqref{eq:er02} allow us to conclude that $b^N \in C_T^\frac12 L^\infty$.

It remains to show that $b_x^N(t, \cdot)$ exists and is bounded uniformly in $t\in[0,T]$. The derivative $b^N_x(t, \cdot)$ is well defined for all $t\in[0,T]$ because $b^N \in C_T\hz{\gamma}$ for all $\gamma>0$.  Using the equivalent norm \eqref{eq:equiv_norm01} and the Bernstein's inequality \eqref{eq:bernstein_inequality} we have
\begin{equation}
		\label{eq:er03}
		\left\|  b^{N}_x \right\|_{\infty, L^{\infty}}
		\leq
		\left\|  b^{N}_x \right\|_{\ctc{\epsilon}}
		\leq
		c \left\|b^{N}\right\|_{\ctc{\epsilon+1}}
		\leq
		c N^{\frac{\epsilon + \hat{\beta} + 1}{2}} \left\| b \right\|_{\ctc{-\hat{\beta}}},
	\end{equation}
    {where the last inequality holds by Lemma \ref{lemma:schauder_estimates}}. Inequalities \eqref{eq:er02} and \eqref{eq:er03} show that $b^N \in L^\infty_T C^1_b$.
It remains to insert the bounds derived above into formulas for $A^N$ and $B^N$ from Proposition~\ref{pr:Euler_conv}.
\end{proof}

Before stating the main result, we state another auxiliary result whose proof is the main content of Section \ref{sec:proof} below. 

\begin{proposition}\label{prop:reg_to_original}
Under Assumption  \ref{ass:b_betahat}, for any $ \hat \alpha \in (1/2, 1 - \hat\beta) $ and any $\beta \in (\hat \beta, 1/2)$, there is a constant $c$ such that
	\begin{equation*}
		\sup_{t \in [0, T]} \mathbbm{E} [ |X^N_t - X_t| ] \leq c \|b^N - b\|_{\ctc{-\beta}}^{2 \hat \alpha  - 1} 
	\end{equation*}
for all $N$ sufficiently large so that $\|b^N - b\|_{\ctc{-\beta}} < 1$. 
\end{proposition}

The approximation error of our numerical scheme comes from two sources: the time discretisation error {of the Euler-Maruyama scheme}, which depends on $m$ and $N$, and the smoothing error coming from replacing $b$ with $b^N$ in the SDE, which depends on $N$. We will now show how to balance those two sources of errors and bound the resulting convergence rate. To this end, we parametrise $N$ in terms of $m$ and {choose this parametrisation to be of the form $N(m) = m^{\eta}$} for some $\eta > 0$. We denote
\[
 \hat X^m := X^{N(m)m}
\]
and consider the strong error
\[
\Upsilon(m) = \sup_{0 \leq t \leq T}
		\expectation
		\left[ \left| \hat X^{m}_{t} - X_{t} \right| \right].
\]
Take any $\epsilon \in (0, 1/2-\hat\beta)$, $\beta \in (\hat \beta, 1/2)$ and ${\hat \alpha} \in (1/2, 1 - \hat\beta)$.
By triangle inequality, we have
	\begin{equation}
		 \label{eq:er_constants}
		\begin{aligned}
			\Upsilon(m) &\leq  \sup_{0 \leq t \leq T}
		\expectation \left[\left| \hat X^{m}_{t} - X^{N(m)}_{t}\right|\right] 
		+
		\sup_{0 \leq t \leq T}
		\expectation \left[\left| X^{N(m)}_{t} - X_{t}\right|\right]\\
   &\leq A^{N(m)} m^{-1} + B^{N(m)} m^{-1/2} + c\left\|b^{N(m)} - b \right\|_{\ctc{-\beta}}^{2\hat \alpha - 1}\\
     & \leq c\Big[
    m^{\eta\frac{\epsilon + \hat{\beta}}{2}} \left\| b \right\|_{\ctc{-\hat{\beta}}} \left( 1+ m^{\eta\frac{\epsilon + \hat{\beta} + 1}{2}} \left\| b \right\|_{\ctc{-\hat{\beta}}} \right) m^{-1}\\
    &\hspace{20pt}+ \Big(
   m^{\eta\frac{\epsilon + \hat{\beta} + 1}{2}} \left\| b \right\|_{\ctc{-\hat{\beta}}} +m^{\eta\frac{\epsilon+\hat\beta}2} \|b\|_{C_T^\frac12 \hz{-\hat\beta}}
    \Big) m^{-\frac12} \\
    &\hspace{20pt} + m^{-\eta\frac{\beta-\hat\beta}2({2\hat\alpha}-1)} \left\| b \right\|_{\ctc{-\hat\beta}}^{{2\hat\alpha} - 1}
    \Big], 
    	\end{aligned}
	\end{equation}
    where in the second inequality {we used Proposition \ref{pr:Euler_conv}}, {Proposition \ref{prop:reg_to_original}}, and in the last inequality Corollary \ref{cor:Euler_conv},  and the following estimate arising from Lemma \ref{lemma:schauder_estimates}:
\[
\|b^{N(m)}(t, \cdot) - b(t, \cdot) \|_{\hz{-\beta}} \le c N(m)^{-\frac{\beta-\hat\beta}{2}} \|b(t, \cdot)\|_{\hz{-\hat\beta}}, \qquad t \in [0, T].
\]

Since all the norms appearing on the right hand side are finite and independent of $m$, they can be absorbed by a constant and we have
\begin{equation}
	\label{eq:subs_m_eta}
	\Upsilon
	\leq
	c
	{\bigg[}
	{m}^{\eta \frac{\epsilon + \hat{\beta}}{2} - 1}
	+
	{m}^{\eta \frac{2\epsilon + 2\hat{\beta} + 1}{2} - 1}
	\\
	+
	{m}^{\eta \frac{\epsilon + \hat{\beta} + 1}{2} - \frac{1}{2}}
	+
    {m}^{\eta \frac{\epsilon + \hat{\beta}}{2} - \frac{1}{2}}
    + 
    m^{ -\eta \frac{(\beta - \hat{\beta})({2 \hat\alpha} - 1)}{2} }
	{\bigg]}.
\end{equation} 
Before proceeding further, we optimise the last term in $\beta$ and ${\hat\alpha}$ to maximise its rate of decrease of the last term in $m$. Recalling the constraints for $\beta$ and ${\hat\alpha}$, the product $(\beta - \hat{\beta})({2 \hat\alpha} - 1)$ is maximised for ${\hat\alpha} \approx 1-\hat\beta$ and $\beta \approx 1/2$. We take ${\hat\alpha} = 1 - \hat\beta -\epsilon$ and $\beta = 1/2 - \epsilon$ which yields the value $2(1/2-\hat\beta - \epsilon)^2$. The last term of \eqref{eq:subs_m_eta} takes the form
\[
 m^{ -\eta  (\frac12 - \hat \beta - \epsilon)^2}.
\]

The monotonicity of the remaining four terms in \eqref{eq:subs_m_eta} depends on $\eta$. For the scheme to converge, we need to make sure that {they all decrease, which is guaranteed} if $\eta\frac{2\epsilon+2\hat\beta+1}2-1<0$ and $\eta\frac{\epsilon+\hat\beta+1}2-\frac12<0$.
This leads to the constraint  
\begin{equation}\label{eq:contraint_eta}
0 < \eta < \frac{1}{\epsilon + \hat{\beta} + 1}. 
\end{equation} 
At this point we have to find the optimal value of $\eta$. It is easy to see that the slowest decreasing term within the first four terms of \eqref{eq:subs_m_eta} is ${m}^{\eta \frac{\epsilon + \hat{\beta} + 1}{2} - \frac{1}{2}}$. To balance Euler-Maruyama error measured by the first four terms and the error of approximating $b$ with $b^{N(m)}$ in the last term we equate the rates
\begin{equation}\label{eqn:eta_equation}
    {\eta \frac{\epsilon + \hat{\beta} + 1}{2} - \frac{1}{2}} =-\eta  (\frac12 - \hat \beta - \epsilon)^2.
    \end{equation}
This leads to 
\begin{equation}\label{eq:eta}
    \eta = \frac1{\epsilon + \hat{\beta} + 1 + 2(\frac12 - \hat\beta -\epsilon)^2 }.
\end{equation}    
Inserting this expression for $\eta$ into the right-hand side of \eqref{eqn:eta_equation} we obtain the rate of convergence of our scheme as 
\[
\Big(\frac{\epsilon + \hat{\beta} + 1}{(\frac12 - \hat\beta -\epsilon)^2} + 2 \Big)^{-1}.
\]

We summarise the above derivation in the following theorem.

\begin{theorem}
\label{th:convergence_rate_es}
Let Assumption \ref{ass:b_betahat} hold and fix $\epsilon \in (0, 1/2-\hat\beta)$. By taking
\[
N(m) = m^{\frac1{\epsilon + \hat{\beta} + 1 + 2(\frac12 - \hat\beta -\epsilon)^2 }},
\]
the strong error of our scheme is bounded as follows: 
	\begin{equation}
		\label{eq:euler_rate}
		\sup_{0 \leq t \leq T}
		\expectation
		\left[ \left| X^{N(m) m}_{t} - X_{t} \right| \right]
		\leq
		c m^{-\big(\frac{\epsilon + \hat{\beta} + 1}{(1/2 - \hat\beta -\epsilon)^2} + 2 \big)^{-1}},
	\end{equation}
    where constant $c$ depends on $\epsilon$ {and the drift $b$.}
\end{theorem}

\begin{remark}\label{rm:8}
Note that the bound stated in Proposition \ref{prop:reg_to_original} holds for $N$ large enough so that $\|b^{N(m)} - b\|_{C_T\hz{-\beta}} < 1$. Formally, the error for small $N(m)$, i.e., small $m$, could be incorporated into the constant $c$ in \eqref{eq:euler_rate} when the condition $\|b^{N(m)} - b\|_{C_T\hz{-\beta}} < 1$ is not satisfied. However, when doing numerical estimation of the convergence rate
(see Section \ref{sec:numerics}) we have to pick $m$ large enough so that $\|b^{N(m)} - b\|_{C_T\hz{-\beta}} < 1$, otherwise the numerical estimate of the rate could not be reliable.
\end{remark}

\begin{corollary}\label{cor:rate}
Our construction allows us to achieve any strong convergence rate strictly smaller than
\[
\lim_{\epsilon \downarrow 0} \Big(\frac{\epsilon + \hat{\beta} + 1}{(\frac12 - \hat\beta -\epsilon)^2} + 2 \Big)^{-1}
=
\Big(\frac{\hat{\beta} + 1}{(\frac12 - \hat\beta)^2} + 2 \Big)^{-1} =: r(\hat \beta).
\]
\end{corollary}

\begin{remark}\label{rm:rate}
We rewrite $r(\hat \beta) =\frac{\left(\frac{1}{2} - \hat{\beta}\right)^2}{2\left(\frac{1}{2} - \hat{\beta}\right)^2 + \hat{\beta} + 1}$ from Corollary \ref{cor:rate}.
Figure \ref{fig:empirical_rate} in Section \ref{sec:numerics} displays this function over the range of $\hat \beta \in (0,\frac{1}{2})$. We take a closer look at the limits as $\hat \beta$ approaches $0$ and $\frac{1}{2}$:
\begin{itemize}
    \item $\lim_{\hat\beta \downarrow 0} r(\hat \beta) = \frac16$. This corresponds to $b\in C_T \hz{0+}$, which is comparable to the case of a measurable function $b\in C_T \hz{0}$. 
  \item $\lim_{\hat\beta \uparrow \frac{1}{2}} r(\hat \beta) = 0$, so when the roughness of the drift approaches the boundary $\frac{1}{2}$, the scheme deteriorates. This is expected due to the nature of the estimate in Proposition \ref{prop:reg_to_original} that we use. 
\end{itemize}
A direct comparison of our rate with other rates found in the literature is only possible in the case of measurable drifts  with Brownian noise, which has been treated in \cite{butkovskyApproximationSDEsStochastic2021}. The rate obtained there is $\frac{1}{2}$, so our estimate is clearly not optimal because we obtain $\frac{1}{6}$. 
However, the technique they use {is different from ours, and in particular they use the stochastic sewing lemma} from \cite{leStochasticSewingLemma2020} to drive up the rate, but it seems it is not straightforward to apply  techniques used in \cite{butkovskyApproximationSDEsStochastic2021} to our setting.

{Paper \cite{goudenegeNumericalApproximationSDEs2022} considers time-homogeneous distributional drifts and, between others, covers drifts} $b\in C^{-\beta}$ for some $\beta>0$, but the driving noise is a fractional Brownian motion with Hurst index $H\in(0,\frac{1}{2})$. For $\beta \in (0,\frac1{2H}-1)$, the rate of convergence is $\frac{1}{2(1+\beta)}-\epsilon$ for any $\epsilon > 0$, which excludes the case  $H=\frac{1}{2}$ of Brownian noise which would lead to a rate of $1/2$ for measurable functions (since $\beta$ would also approach 0). 

In \cite{deangelisNumericalSchemeStochastic2022}, the authors consider SDE \eqref{eq:sde_S} with the drift $b \in C^{\kappa}_T H^{-\beta}_{{(1-\beta)}^{-1}, q}$ for $\kappa \in (\frac{1}{2}, 1)$, $\beta \in (0, \frac{1}{4})$ and $q \in (4, \frac{1}\beta)$. The space $ H^{-\beta}_{{(1-\beta)}^{-1}, q}$ is a fractional Sobolev space of negative regularity $-\beta$. Their $\beta$ is closely related to our $\hat\beta$ as fractional Sobolev spaces $H^{s}_{p,q}$ are related, albeit different, to Besov spaces $B^s_{p,q}$, see e.g., \cite{triebelTheoryFunctionSpaces2010}. In both cases the index $s$ is a measure of smoothness of the elements of the space. When $\beta \downarrow 0$, the convergence rate of the numerical scheme in \cite{deangelisNumericalSchemeStochastic2022} tends to $\frac{1}{6}$, as in our case. However, when $\beta\to \frac{1}{4}$ the convergence rate in \cite{deangelisNumericalSchemeStochastic2022} vanishes, while we get $r(\frac{1}{4}) = {\frac{1}{22}}$.

Finally, we would like to mention that in our numerical study in Section \ref{sec:numerics} we obtain a numerical estimate of the convergence rate equal to  {$\frac{1+(-\hat \beta)}{2}$,}  which is the equivalent of the rate {$\frac{1+\gamma}{2}$} found in \cite{butkovskyApproximationSDEsStochastic2021} for positive $\gamma$, i.e. for $\gamma$-H\"older continuous drifts. This suggests that {the rate from the latter paper could apply {in our setting, but it would require} a different approach and is left as future work.} 
\end{remark}


\section{Proof of Proposition \ref{prop:reg_to_original}}\label{sec:proof}

Let us introduce an auxiliary process $Y^N$, which is the (weak) solution of the SDE
\begin{equation}
	\label{eq:YN_def}
	Y^N_t = y^N_0 + \lambda \int_0^t u^N(s, \psi^N(s, Y^N_s)) ds + \int_0^t (  u_x^N (s, \psi^N(s, Y^N_t)) + 1 ) dW_s,
\end{equation}
where $u^N$ is the unique solution to the regularised Kolmogorov equation 
\begin{equation}
		\label{eq:kolmogorov_N}
		\begin{cases}
		 u^N_t + \frac{1}{2} u^N_{xx} + b^N u_x^N = \lambda u^N - b^N	  	  	\\
	  	  	u^N(T) = 0,
	  	\end{cases}
	\end{equation}
and $\psi^N (t, x)$ is the space-inverse of  
\begin{equation} \label{eq:phiN}
\phi^N(t, x):= x+ u^N (t, x),
\end{equation}
which exists by 
\cite[Proposition 4.16]{issoglioPdeDriftNegative2022}
for $\lambda$ large enough.
Since 
$b^N \to b$ 
in 
$C_T \hz{-\hat\beta}$
by Lemma \ref{lemma:schauder_estimates}
and Assumption 
\ref{ass:b_betahat},
we can apply 
\cite[Lemma 4.19]{issoglioPdeDriftNegative2022}.
This lemma and its proof provide key properties of the above system and its relation to 
\eqref{eq:kolmogorov} and \eqref{eq:Y_def}.
The solution of the regularised Kolmogorov equation 
\eqref{eq:kolmogorov_N}
enjoys a bound 
$\|u^N_x\|_\infty <1/2$ 
for 
$\lambda$
large enough; the constant 
$\lambda$
can be chosen independently of 
$N$
(but it depends on the drift 
$b$
of the original SDE), and from now on we fix such 
$\lambda$.
We also have 
$u^N \to u$
and
$u^N_x \to u_x$
uniformly on 
$[0,T]\times \mathbb R^d$.  

This section is devoted to the proof of Proposition \ref{prop:reg_to_original} with the overall structure inspired by \cite{deangelisNumericalSchemeStochastic2022}. The main idea of the proof is to rewrite the solutions $X^N$ and $X$ in their equivalent virtual formulations and thus study the error between the auxiliary processes $Y^N$ defined in \eqref{eq:YN_def} and $Y$ defined in \eqref{eq:Y_def}. This transformation has the advantage that the SDEs for $Y^N$ and $Y$ are classical SDEs with strong solutions, see Remark \ref{rm:strong}. After writing the difference  $X^N-X$ in terms of $Y^N-Y$ we apply It\^o's formula to $|Y^N-Y|$. We will control the stochastic and Lebesgue integrals by $u^N-u$ and $\psi^N-\psi$, which in turn are bounded by some function of $b^N-b$. The term involving the local time at $0$ of $Y^N-Y$ is handled in Lemma \ref{lemma:local-time-at-0}.

In the remainder of this section, we fix $\beta \in (\hat \beta, 1/2)$ as in the statement of Proposition \ref{prop:reg_to_original}. Since $\beta \in (\hat \beta, 1/2)$, we have $b \in C_T \hz{(-\beta)+}$ by Assumption \ref{ass:b_betahat}. Notice  that the solutions $u$ and $u^N$ obtained when taking $b$ as an element of $ C^{1/2}_T \hz{(-\hat\beta)+}$ are the same as when viewing the same $b$ as an element of $C_T \hz{(-\beta)+}$. We obviously have $u, u^N \in \ctc{1+ \alpha}$ for any $ \alpha < 1- \beta$ but also for any $ \alpha < 1- \hat\beta$.

\begin{lemma}\label{lemma:diff_u_uN}
There are $\rho_0$ and $c$ such that for any $\rho\geq \rho_0$ and for any $\alpha <1-\beta $  we have
\begin{equation*}
	\| u - u^N \|^{(\rho)}_{\ctc{1 + \alpha}}
	\leq
	2c \| b - b^N \|_{\ctc{-\beta}} (\| u \|^{(\rho)}_{\ctc{1 + \alpha}} - 1) \rho^{\frac{\alpha + \beta - 1}{2}}.
\end{equation*}
\end{lemma}
\begin{proof}
The bound in the statement of the lemma forms the main part of the proof of \cite[Lemma 4.17]{issoglioPdeDriftNegative2022} in the special case when the terminal condition is zero and $g^N = b^N$, $g = b$.
\end{proof}

The parameter $\rho$ and the exact dependence of the bound on it is not important for our arguments, so we may fix $\rho$ that satisfies the conditions of the above lemma. However, for completeness of the presentation, we will mention the dependence on $\rho_0$ and $\rho$ in results below.

Using Lemma \ref{lemma:diff_u_uN} we can derive a uniform bound on the $L^\infty$-norm of the difference $u(t) - u^N(t)$ and $u_x(t) - u_x^N(t)$.
\begin{lemma}\label{lemma:diff_uN_graduN}
For any $\rho \ge \rho_0$, where $\rho_0$ is from Lemma \ref{lemma:diff_u_uN}, there is $\kappa > 0$ such that for all $t \in [0, T]$
	\begin{equation*}
		\| u(t) - u^N(t) \|_{L^\infty} + \|  u_x(t) - u^N_x(t) \|_{L^\infty} 
		\leq
		\kappa \| b - b^N \|_{ \ctc{-\beta}}.
	\end{equation*}
    {The constant $\kappa$ depends also on} $c$ from Lemma \ref{lemma:diff_u_uN} and on $T$ and $u$.
\end{lemma}
\begin{proof}
We recall that $u, u^N \in C_T \hz{1+\alpha}$ and $u_x, u^N_x \in C_T \hz{\alpha}$ for any $\alpha<1-\beta$. Fix such $\alpha$. Recall also that the norm in $\hz{\gamma}$ for $\gamma\in(0,1)$ is given by \eqref{eq:equiv_norm01} and in $\hz{\gamma}$ for $\gamma\in(1,2)$ is given by \eqref{eq:equiv_norm12}.  Using this together with the Bernstein inequality \eqref{eq:bernstein_inequality} we get for all $t\in[0,T]$ that 
\[
\| u(t) - u^N(t) \|_{L^\infty} + \|  u_x(t) - u^N_x(t) \|_{L^\infty} \leq c' \| u(t) - u^N(t) \|_{\hz{1+\alpha}} ,
\]
for some $c'>0$.
We conclude using \eqref{eq:rho_eq_norm} and Lemma \ref{lemma:diff_u_uN}. The inequality in the statement of the lemma holds with constant $\kappa$ given by 
\[
\kappa =  c'e^{\rho T}  2c  (\| u \|^{(\rho)}_{\ctc{1 + \alpha}} - 1)  {\rho^{\frac{\alpha + \beta - 1}{2}}}.  \qedhere
\]
\end{proof}
Next we derive a bound for the difference $\psi-\psi^N$, where we recall that the two functions are the space-inverses of $\phi$ and $\phi^N$ defined in \eqref{eq:phi} and \eqref{eq:phiN}. 

\begin{lemma}\label{lemma:bound_psi-psiN}
Take $\kappa$ from Lemma \ref{lemma:diff_uN_graduN}. We have
	\begin{equation}
	    \sup_{(t,y) \in [0,T] \times \R}
	    \left|\psi (t,y) - \psi^N(t,y)\right| \leq 2 \kappa \left\|b - b^N\right\|_{\ctc{-\beta}}.
	\end{equation}
\end{lemma}
\begin{proof}
Let us first recall that $\|u_x\|_\infty  \leq \frac12$, see the discussion at the beginning of this section. Hence $u$ is $\frac12$-Lipschitz. Using this we have for any $x, x' \in \mathbb R$ 
\begin{equation}\label{eq:phi_difference} 
		\begin{aligned}
			\left| \phi(t, x) - \phi(t, x') \right| 
			&= 
			\left| \left( x + u(t, x) \right) - \left( x' + u(t, x') \right) \right|
			\\
			&\geq
	   	  |x - x'| - |u(t, x) - u(t, x')| 
			\\
			&\geq
			  |x - x'| - \frac12 |x-x'|  \\
			&\geq
                \frac12 |x-x'|	.
		\end{aligned}
	\end{equation}
Insert $ x= \psi(t, y) $ and $ x' = \psi^{N}(t, y) $ for some $y \in \R$, to get the bound
\[
\left| \phi(t, \psi(t,y)) - \phi(t, \psi^{N}(t, y)) \right| \geq \frac12 \left| \psi(t, y) - \psi^N(t, y) \right|.
\]
This implies the first inequality below
	\begin{equation}
		\label{eq:psi_phi_1}
		\begin{aligned}
			\left| \psi(t, y) - \psi^{N}(t, y) \right| 
			&\leq 
			2 
			\left| \phi(t, \psi(t, y)) - \phi(t, \psi^{N}(t,y)) \right|
			\\
			&=
			2
			\left| \phi^{N}(t, \psi^{N}(t, y)) - \phi(t, \psi^{N}(t,y)) \right|
   \\
			&=
			2
			\left| u^{N}(t, \psi^{N}(t, y)) - u(t, \psi^{N}(t,y)) \right|,
		\end{aligned}
	\end{equation}
where we used $ \phi^{N}(t, \psi^{N}(t, y)) = y = \phi(t, \psi(t, y)) $ and  the  definition of $\phi$  and $\phi^N$. 
Thus
\[                
\left| \psi(t, y) - \psi^{N}(t, y) \right| 
\leq
2\left\| u(t) - u^{N}(t) \right\|_{L^{\infty}}
\leq
2\kappa	\left\| b - b^{N} \right\|_{C_T\hz{-\beta}},
\]
having used Lemma \ref{lemma:diff_uN_graduN} in the last inequality. 
\end{proof}

Finally we derive a bound for $|u^N(s, \psi^N(s, y)) - u(s, \psi(s, y'))|$ and for $|u_x^N(s, \psi^N(s, y)) - u_x(s, \psi(s, y'))|$.

\begin{lemma}\label{lemma:uN-n_bound_for_integral}
For any  $\alpha< 1-\hat\beta$ and any $y,y' \in \R$, the following bounds are satisfied
     \begin{equation} \label{eq:bound_u_abs}
        \left|u^N(s, \psi^N(s, y')) - u(s, \psi(s, y))\right|
  		\leq 2 \kappa \left\|b^N - b\right\|_{\ctc{-\beta}} + \left|y-y'\right|,
    \end{equation}
and
    \begin{equation}     \label{eq:bound_ux_abs}
            \begin{split}
                \left|u_x^N(s, \psi^N(s, y')) - u_x(s, \psi(s, y))\right|
  		\leq & \kappa \left\| b - b^N \right\|_{\ctc{-\beta}}
            + 2^{\alpha}\kappa^{\alpha} \left\|u\right\|_{\ctc{1 + \alpha}}  \left\|b^N - b\right\|_{\ctc{-\beta}}^{\alpha}   \\
         &+\left|u_x(s, \psi(s, y)) -  u_x(s, \psi(s, y'))\right|,
            \end{split}
    \end{equation}
where $\kappa$ is from Lemma \ref{lemma:diff_uN_graduN}.
\end{lemma}
\begin{proof}
We start with \eqref{eq:bound_u_abs}. We rewrite the left-hand side as
\begin{equation}\label{eq:uN-u}
\begin{aligned}
\left|u^N\left(s, \psi^N\left(s,y'\right)\right) - u\left(s, \psi\left(s, y\right)\right)\right|
&\leq
\left|u^N\left(s, \psi^N\left(s, y'\right)\right) - u\left(s, \psi^N\left(s, y'\right)\right)\right|
\\
&\hspace{11pt}+
\left|u\left(s, \psi^N\left(s, y'\right)\right) - u\left(s, \psi\left(s, y'\right)\right)\right|
\\
&\hspace{11pt}+
\left|u\left(s, \psi\left(s,y'\right)\right) - u\left(s, \psi\left(s, y\right)\right)\right|.
\end{aligned}
\end{equation}
Using Lemma  \ref{lemma:diff_uN_graduN}, we bound the first term above by
$$	
\left|u^N\left(s, \psi^N\left(s, y'\right)\right) - u\left(s, \psi^N\left(s, y'\right)\right)\right|
\leq \|u^N\left(s\right) - u\left(s\right)\|_{L^\infty}  
\leq \kappa \left\| b - b^N \right\|_{\ctc{- \beta}}. 
$$
The second term in \eqref{eq:uN-u} is bounded using the fact that $\|u_x\|_\infty \le \frac12$ (see the discussion at the beginning of this section) and Lemma \ref{lemma:bound_psi-psiN}:
$$
\left|u\left(s, \psi^N\left(s, y'\right)\right) - u\left(s, \psi\left(s, y'\right)\right)\right| 
\leq \frac12 |\psi^N\left(s, y'\right) -   \psi\left(s, y'\right)| 
\leq  \kappa \left\| b - b^N \right\|_{\ctc{-\beta}}.
$$		
For the third term in \eqref{eq:uN-u}, we use again that $u(t, \cdot)$ is $\tfrac12$-Lipschitz 
$$		
 \left|u\left(s, \psi\left(s,y'\right)\right) - u\left(s, \psi\left(s, y\right)\right)\right|
 \leq \frac12|\psi\left(s,y'\right) - \psi\left(s, y\right)| 
 \leq  |y-y'|,
$$	
where the last inequality follows from the fact that $\psi(t, \cdot)$ is $2$-Lipschitz, which we argue as follows. By the definition of $\psi$ as the space inverse of $\phi$, we have $z = \psi(t,z) + u(t, \psi(t,z))$. Hence
\begin{align*}
\left| \psi(t, y) - \psi(t, y') \right| 
&\leq 
|u(t, \psi(t,y)) - u(t, \psi(t, y')| + |y - y'|
\le
\frac12 | \psi(t, y) - \psi(t, y')| + |y - y'|,
\end{align*}
where the last inequality uses again that $u(t, \cdot)$ is $\frac12$-Lipschitz. Combining the above estimates proves \eqref{eq:bound_u_abs}.

Let us now prove \eqref{eq:bound_ux_abs}. Similarly as above we write
\begin{equation}\label{eq:uNx-ux}
		\begin{aligned}
			\left| u^N_x(s, \psi^N(s, y')) -  u_x(s, \psi(s, y))\right|
	  	  	&\leq
	  	  	\left| u_x^N(s, \psi^N(s, y')) - u_x(s, \psi^N(s, y'))\right|
	  	  	\\
	  	  	&\hspace{11pt}+
	  	  	\left| u_x(s, \psi^N(s, y')) -  u_x(s, \psi(s, y'))\right|
	  	  	\\
	  	  	&\hspace{11pt}+
	  	  	\left| u_x(s, \psi(s, y')) - u_x(s, \psi(s, y))\right|.
	  	\end{aligned}
    \end{equation}
The first term is bounded using Lemma \ref{lemma:diff_uN_graduN} as follows
$$	
\left|u^N_x\left(s, \psi^N\left(s, y'\right)\right) - u_x\left(s, \psi^N\left(s, y'\right)\right)\right|
\leq \|u_x^N\left(s\right) - u_x\left(s\right)\|_{L^\infty}  
\leq \kappa \left\| b - b^N \right\|_{{\ctc{-\beta}}}.
$$
Since $b \in C_T \hz{-\hat\beta}$ and $\alpha \in (1/2, 1-\hat\beta)$ then $u\in C_T \hz{1+ \alpha}$, so by the Bernstein inequality (Lemma \ref{lemma:bernstein_inequality}) we have $u_x(t) \in C_T \hz{{\alpha}}$ with $\|u_x\|_{{C_T\hz{\alpha}}} \le \|u\|_{{C_T\hz{1+\alpha}}}$. Recalling the norm \eqref{eq:equiv_norm01} in $\hz{\alpha}$, we conclude that $u_x(t)$ is $\alpha$-H\"older continuous with constant $\|u\|_{{C_T \hz{1+\alpha}}}$. This allows us to bound the second term in  \eqref{eq:uNx-ux} as
\begin{align*}
\left|u_x(s, \psi^N(s, y')) - u_x\left(s, \psi\left(s, y'\right)\right)\right| 
&\leq \|u\|_{{C_T \hz{1+\alpha}}} |\psi^N\left(s, y'\right) -   \psi\left(s, y'\right)|^{\alpha}  \\
&\leq \|u\|_{C_T \hz{1+\alpha}} 2^{\alpha}\kappa^{\alpha}     \left\| b - b^N \right\|_{\ctc{-\beta}}^{\alpha},
\end{align*}
where the last inequality uses Lemma \ref{lemma:bound_psi-psiN}. This concludes the proof.   
\end{proof}

\begin{remark}\label{rem:L1}
    Note that in our main result we find the strong error in $L^1$ instead of the most commonly used estimates in $L^2$,
    the reason for this is illustrated here.
    Let us apply It\^o's formula to
    $(Y^N_t - Y_t)^2$, integrate from $0$ to $t$ and take expectation.
    Apply further Lemma \ref{lemma:uN-n_bound_for_integral}
    and recall that
    $\|u_x\|_\infty < 1/2$,
    $\psi$
    is
    2-Lipschitz,
    and
    $u_x$
    is
    $\alpha$-H\"older
    continuous for any
    $\alpha < 1 - \beta$.
    We get
    \begin{align*}
       \mathbb E (Y^N_t - Y_t )^2 - (Y^N_0 - Y_0)^2
       &\leq 2\mathbb E \int_0^t |Y^N_s - Y_s|^2 
       +  4 \kappa \|b^N - b\|_{C_T \mathcal{C}^{-\beta}} \mathbb E \int_0^t |Y^N_s - Y_s| ds
       \\
       &\phantom{=}+ t(\kappa \| b^N - b \|_{C_T \mathcal{C}^{-\beta}} + 2^\alpha \kappa^\alpha \|u\|_{C_T \mathcal{C}^{-\beta}}\|b^N - b\|_{C_T \mathcal{C}^{-\beta}}^\alpha)^2
       \\
       &\phantom{=}+ \mathbb E \int_0^t c^2 2^{2\alpha} |Y^N_s - Y_s|^{2\alpha} ds.
    \end{align*}
    Note that the second and last terms in the inequality have the difference $|Y^N_s-Y_s|$ in smaller powers than $2$ (the latter since $\alpha < 1$), so they decrease slower than the term on the left-hand side and Gronwall inequality cannot be applied. The power $2\alpha$ in the last term follows from the fact that $u_x$ is $\alpha$-H\"older continuous which is an inherent feature of the construction of the virtual solution to \eqref{eq:sde_S}.
\end{remark}

In order to bound the $L^1$ norm of the difference $X^N - X$, we need to bound the local time of $Y^N - Y$ at zero. To this end, we recall the definition of a local time of a continuous semimartingale $Z$ and a bound for this local time established in \cite{deangelisNumericalSchemeStochastic2022}. We define the local time of $Z$ at $0$ by
  	\begin{equation}\label{eq:local_time_zero}
		L^0_t (Z) = \lim_{\epsilon \to 0} \frac{1}{2 \epsilon} \int_0^t \mathbbm{1}_{\{| Z | \leq \epsilon\}} d \langle Z \rangle_s,  	  	\qquad t \geq 0.
  	\end{equation}
\begin{lemma}[{\cite[Lemma 5.1]{deangelisNumericalSchemeStochastic2022}}]
    \label{lemma:local-time-at-0}
  	For any
  	$\epsilon \in (0,1)$
  	and any real-valued, continuous semi-martingale
  	$Z$
  	we have
  	\begin{equation}
		\expectation[L^0_t(Z)] \leq 
  	  	  4 \epsilon
  	  	- 2 \expectation \left[ \int_0^t \left( \mathbbm{1}_{\{Z_s \in (0, \epsilon)\}} + \mathbbm{1}_{\{Z_s \geq \epsilon\}} e^{1 - Z_s/\epsilon} \right) dZ_s \right]
  	  	+ \frac{1}{\epsilon} \expectation \left[ \int_0^t \mathbbm{1}_{\{Z_s > \epsilon\}} e^{1 - Z_s/\epsilon} d \langle Z \rangle_s \right].
  	\end{equation}
\end{lemma}

The next Proposition is a bound for the local time of $ Y^N-Y$, solutions to the SDEs  \eqref{eq:Y_def} and \eqref{eq:YN_def}.
{Due to the application of  It\^o's formula to $|Y^N-Y|$, this is a key bound in the proof of the $L^1$-convergence of $X^N $ to $X$.}
\begin{proposition}\label{prop:bound_local_time_sde}
For any	$\alpha \in (1/2, 1 - \hat\beta)$ and $N$ such that $\|b^N - b\|_{{C_T\hz{-\beta}}} < 1$ we have
	\begin{equation}
	\label{eq:local_time_YNY_bound}
	    \begin{aligned}
			\expectation
	    	[L^0_t(Y^N - Y)]
	    	&\leq
	    	A \, \expectation
	    	\left[
	    	\int_0^t |Y^N_s - Y_s| ds
	    	\right]
	    	+
	    	B \left\|b^N - b\right\|_{{\ctc{-\beta}}}^{{2 \alpha} - 1}
	    	+ o\left(\left\|b^N - b\right\|_{\ctc{-\beta}}^{ {2 \alpha - 1}}\right),
	    \end{aligned}
	\end{equation}
for some constants
$ A, B > 0$.
\end{proposition}

\begin{proof}
This proof follows the same steps as the proof of \cite[Proposition 5.4]{deangelisNumericalSchemeStochastic2022} with differences coming from the spaces that $b$ and $b^N$ belong to. It is provided for the reader's convenience.

Recall that $Y_{t}, Y^{N}_{t}$ are strong solutions of \eqref{eq:Y_def}	and \eqref{eq:YN_def} respectively, see Remark \ref{rm:strong}. Their difference satisfies
\begin{equation}\label{eq:YN-Y}
		\begin{aligned}
			Y^N_t - Y_t 
		&=
			(y^N_0 - y_0) + \lambda \int_0^t ( u^N(s, \psi^N(s, Y^N_s)) - u(s, \psi(s, Y_s)) ) ds
			\\
			&\hspace{11pt}+
			\int_0^t (  u_x^N (s, \psi^N(s, Y^N_s)) -  u_x (s, \psi(s, Y_s))) dW_s.
		\end{aligned}
	\end{equation}
We apply Lemma \ref{lemma:local-time-at-0} to $Y^N_t - Y_t$ for $\epsilon \in (0,1)$:
\begin{align}
&\expectation [L_t^0 (Y^N - Y)] \leq
		4 \epsilon\notag\\
&
		-
		2\lambda
		\expectation
		\left[
			\int_0^t
			\left( 
			\mathbbm{1}_{\{Y^N_s - Y_s \in (0, \epsilon)\}} 
			+ \mathbbm{1}_{\{Y^N_s - Y_s \geq \epsilon\}} e^{1 - \frac{Y^N_s - Y_s}{\epsilon}}
			\right)
			\left(
			u^N(s, \psi^N(s, Y^N_s)) - u(s, \psi(s, Y_s))
			\right) 
			ds
		\right]
		\label{eq:local_time_diff_u}
		\\
		&+
		\frac{1}{\epsilon}
		\expectation
		\left[
		\int_0^t \mathbbm{1}_{\{ Y^N_s - Y_s > \epsilon \}} e^{1 - \frac{Y^N_s - Y_s}{\epsilon}}
		\left(  u_x^N(s, \psi^N(s, Y^N_s)) -  u_x(s, \psi(s, Y_s)) \right)^2 ds
		\right],
		\label{eq:local_time_diff_gradu}
	\end{align}
where the expectation of the integral with respect to the Brownian motion $(W_t)$ is zero thanks to the fact that $u_x$ and $u_x^N$ are bounded uniformly in $N$.
	
Notice that if 	$Y^N_s - Y_s \geq \epsilon$, then $e^{1 - \frac{Y^N_s - Y_s}{\epsilon}} \leq 1$. Hence, \eqref{eq:local_time_diff_u} {is bounded from above by}
\[
4\lambda  \expectation
		\left[
			\int_0^t \big(u(s, \psi(s, Y_s)) - u^N(s, \psi^N(s, Y^N_s))\big) ds
		\right]
\le
  	        4 \lambda \left(
  	        2\kappa \left\|b^N - b\right\|_{{\ctc{-\beta}}} t 
  	        + \expectation\left[\int_0^t\left|Y^N_s - Y_s\right| ds \right]
  	    	\right),
\]
where the last inequality is by Lemma \ref{lemma:uN-n_bound_for_integral}.

Now for \eqref{eq:local_time_diff_gradu}, we use again the observation that $\mathbbm{1}_{\{ Y^N_s - Y_s > \epsilon \}} e^{1 - \frac{Y^N_s - Y_s}{\epsilon}} \le 1$, the estimate \eqref{eq:bound_ux_abs} from Lemma \ref{lemma:uN-n_bound_for_integral} choosing $\alpha' \in (\alpha,  1-\hat \beta)$, and the inequality $(x_1 + x_2 + x_3)^2 \leq 3(x_1^2 + x_2^2 + x_3^2)$ to get the bound
\begin{equation*}
    \begin{aligned}
&			\frac{1}{\epsilon}
        \expectation
        \int_0^t
        \left(
        3 \kappa^2 \left\| b - b^N \right\|_{{\ctc{-\beta}}}^2
        + 3 \cdot 2^{2\alpha'} \kappa^{2\alpha'}
        \left\|b^N - b\right\|_{{\ctc{-\beta}}}^{2\alpha'}
        \left\|u\right\|_{\ctc{1 + \alpha'}}^2
        \right)ds
        \\
        &\qquad+
        \frac{1}{\epsilon}
        \expectation
        \int_0^t
        3 \mathbbm{1}_{\left\{Y^N_s - Y_s > \epsilon\right\}} 
        e^{1 - \frac{Y^N_s - Y_s}{\epsilon}} 
        \left| u_x(s, \psi(s, Y^N_s)) - u_x(s, \psi(s, Y_s))\right|^2 
        ds
        \\
        &\leq
        \frac{3 t}{\epsilon}  \left\| b^N - b\right\|_{{\ctc{-\beta}}}
        \left(
        \kappa^2 \left\|b^N - b\right\|_{{\ctc{-\beta}}}
        +
        (2 \kappa)^{2 \alpha'} \left\|u\right\|_{\ctc{1 + \alpha'}}^2 \left\|b^N - b\right\|_{{\ctc{-\beta}}}^{2 \alpha' - 1}
        \right)
        \\
        &\qquad+
        \frac{3}{\epsilon} \expectation
        \left( 
            \int_0^t 
            \mathbbm{1}_{\left\{ Y^N_s - Y_s > \epsilon \right\}} 
            e^{1 - \frac{Y^N_s - Y_s}{\epsilon}} 
            \left| u_x(s, \psi(s, Y^N_s)) - u_x(s, \psi(s, Y_s)) \right|^2 
            ds
        \right).
    \end{aligned}
\end{equation*}
Putting the above estimates together yields the following bound for the expectation of the local time: 
\begin{equation}\label{eqn:localtime}
\begin{aligned}
		\expectation [L_t^0 (Y^N - Y)]
		&\leq
		4 \epsilon + 	 4 \lambda \left(
  	        2\kappa \left\|b^N - b\right\|_{{\ctc{-\beta}}} t 
  	        + \expectation\left[\int_0^t\left|Y^N_s - Y_s\right| ds \right] \right)\\
           &\hspace{11pt}+\frac{3t}{\epsilon} \left\| b^N - b\right\|_{{\ctc{-\beta}}}
	    	\left(
	    	\kappa^2 \left\|b^N - b\right\|_{{\ctc{-\beta}}}
	    	+
	    	(2 \kappa)^{2 \alpha'} \left\|u\right\|_{\ctc{1 + \alpha'}}^2 \left\|b^N - b\right\|_{{\ctc{-\beta}}}^{2 \alpha' - 1}
	    	\right)
	    	\\
			&\hspace{11pt}+
	    	\frac{3}{\epsilon} \expectation
	    	\left[ 
				\int_0^t 
				\mathbbm{1}_{\left\{ Y^N_s - Y_s > \epsilon \right\}} 
				e^{1 - \frac{Y^N_s - Y_s}{\epsilon}} 
				\left| u_x(s, \psi(s, Y^N_s)) -  u_x(s, \psi(s, Y_s)) \right|^2 
				ds
	    	\right].
  \end{aligned}
\end{equation}
We take $\epsilon = \|b^N - b\|_{{C_T\hz{-\beta}}} < 1$ and choose $\zeta \in (0,1)$ such that $\alpha'\zeta > 1/2$. Recall that $u_x(s, \cdot)$ is $\alpha'$-H\"older continuous with the constant $\|u\|_{C_T\hz{1+\alpha'}}$, as $\alpha'<1-\hat\beta$, (see Section \ref{subsec:soln_sde}) and $\psi(s, \cdot)$ is $2$-Lipschitz, so
\[
\left| u_x(s, \psi(s, Y^N_s)) -  u_x(s, \psi(s, Y_s)) \right| 
\le
\|u\|_{C_T\hz{1+\alpha'}} 2^{\alpha'} |Y^N_s - Y_s|^{\alpha'}.
\]
This bound and the observation that $\epsilon^\zeta > \epsilon$ yield
\begin{align*}
&\frac{3}{\epsilon} \expectation \left[ \int_0^t 
\mathbbm{1}_{\left\{ Y^N_s - Y_s > \epsilon \right\}} 
e^{1 - \frac{Y^N_s - Y_s}{\epsilon}} \left| u_x(s, \psi(s, Y^N_s)) -  u_x(s, \psi(s, Y_s)) \right|^2 ds \right]\\
&\le
\frac{3}{\epsilon} \expectation \left[ \int_0^t 
\mathbbm{1}_{\left\{ \epsilon < Y^N_s - Y_s \le \epsilon^\zeta \right\}} 
e^{1 - \frac{Y^N_s - Y_s}{\epsilon}} \left| u_x(s, \psi(s, Y^N_s)) -  u_x(s, \psi(s, Y_s)) \right|^2 ds \right]\\
&+
\frac{3}{\epsilon} \expectation \left[ \int_0^t 
\mathbbm{1}_{\left\{ Y^N_s - Y_s > \epsilon^\zeta \right\}} 
e^{1 - \frac{Y^N_s - Y_s}{\epsilon}} \left| u_x(s, \psi(s, Y^N_s)) -  u_x(s, \psi(s, Y_s)) \right|^2 ds \right]\\
&\le
\frac{3}{\epsilon} t \|u\|^2_{C_T\hz{1+\alpha'}} 2^{2\alpha'} \epsilon^{2\alpha'\zeta} + \frac{3}{\epsilon} t e^{1-\epsilon^{\zeta-1}},
\end{align*}
where in the last inequality we used that $\|u_x\|_{L^\infty} \le 1/2$. We insert this bound into \eqref{eqn:localtime} and recall that $\epsilon = \|b^N - b\|_{{C_T\hz{-\beta}}}$ to obtain
\begin{align*}
&		\expectation [L_t^0 (Y^N - Y)]\\
		&\leq
		(4 + 8 \lambda \kappa t + 3\kappa^2 t) \left\|b^N - b\right\|_{\ctc{-\beta}} + 	 4 \lambda \left(
 	        \expectation\left[\int_0^t\left|Y^N_s - Y_s\right| ds \right] \right)\\ 
           &\hspace{11pt}+3t
	    	\left(
	    	(2 \kappa)^{2 \alpha'} \left\|u\right\|_{\ctc{1 + \alpha'}}^2 \left\|b^N - b\right\|_{{\ctc{-\beta}}}^{2 \alpha' - 1}
	    	\right)
	    	\\
			&\hspace{11pt}+
	    	3 t \|u\|^2_{C_T\hz{1+\alpha'}} 2^{2\alpha'} \left\|b^N - b\right\|_{{\ctc{-\beta}}}^{2\alpha'\zeta - 1} 
	    	+ \frac{3t}{\left\|b^N - b\right\|_{{\ctc{-\beta}}}} e^{1-\left\|b^N - b\right\|_{{\ctc{-\beta}}}^{\zeta-1}}\\
			&\le
			3 T \|u\|^2_{C_T\hz{1+\alpha'}} 2^{2\alpha'} \left\|b^N - b\right\|_{{\ctc{-\beta}}}^{2\alpha'\zeta - 1}
			+ 4 \lambda \left(\expectation\left[\int_0^T\left|Y^N_s - Y_s\right| ds \right] \right)
			+ o \left(\left\|b^N - b\right\|_{{\ctc{-\beta}}}^{2\alpha'\zeta - 1}\right).
\end{align*}
Taking $\zeta =  \alpha / \alpha'$ completes the proof.
\end{proof}

\begin{proof}[Proof of Proposition \ref{prop:reg_to_original}]
    Our arguments are inspired by the proof of \cite[Proposition 3.1]{deangelisNumericalSchemeStochastic2022}. For our arguments we fix $\rho$ {that satisfies the conditions} of Lemma \ref{lemma:diff_u_uN} and denote by $\kappa$ the constant from Lemma \ref{lemma:diff_uN_graduN}. By the definition of $ \psi, \psi^N $, {Lemma \ref{lemma:bound_psi-psiN}, and the fact that} $\psi$ is 2-Lipschitz, 	we have
\begin{equation}\label{eq:|XN-X|}
\begin{aligned}
|X^N_t - X_t| &=	| \psi^N(t, Y_t^N) - \psi(t, Y_t)| \leq | \psi^N(t, Y_t^N) - \psi(t, Y_t^N)| + | \psi(t, Y_t^N) - \psi(t, Y_t)|\\ &\leq 
2 \kappa \| b^N - b \|_{\ctc{-\beta}} + 2 |Y^N_t - Y_t|.
\end{aligned}
\end{equation}
For the term $|Y^N - Y|$ we use It\^o-Tanaka's formula and take expectations of both sides:
	\begin{equation*}
		\begin{aligned}
			\mathbbm{E}\left[|Y^N_t - Y_t|\right]
			&= \mathbbm{E}|Y^N_0 - Y_0| + \mathbbm{E}\left[\frac{1}{2} L^0_t (Y^N - Y)\right]
			\\
			&\hspace{11pt}+ \lambda \mathbbm{E}\left[\int_0^t \sgn (Y^N - Y)(u^N(s, \psi^N(s, Y^N_s)) - u(s, \psi(s, Y_s))) ds\right],
		\end{aligned}
	\end{equation*}
    where the stochastic integral disappears due to the boundedness of $u_x$ and $u^N_x$. For the first term we use that $Y_0 = x + u(0, x)$, {$Y^N_0 = x + u^N(0, x)$, and Lemma \ref{lemma:diff_u_uN} with any $0<\alpha<1-\beta$, to conclude that}
\[
|u^N(0,x)-u(0,x)| \le \|u^N - u\|_{\ctc{1+\alpha}}^{(\rho)} \leq o\left(\left\|b^N - b\right\|_{\ctc{-\beta}}^{2 \hat\alpha - 1}\right),
\]
where we used that $\left\|b^N - b\right\|_{\ctc{-\beta}} = o\left(\left\|b^N - b\right\|_{\ctc{-\beta}}^{2 \hat\alpha - 1}\right)$, because $2\hat\alpha - 1 < 1$ and $\left\|b^N - b\right\|_{\ctc{-\beta}} < 1$ as assumed in the statement of the proposition. The second term is bounded by Proposition \ref{prop:bound_local_time_sde} with $\alpha= \hat \alpha$. For the third term we employ a bound from Lemma \ref{lemma:uN-n_bound_for_integral} again with $\alpha= \hat \alpha$ and the fact that $|\sgn(x)| \leq 1$. In summary, we obtain
 \begin{equation*}
\mathbbm{E}\left[|Y^N_t - Y_t|\right]
\leq
o\left(\left\|b^N - b\right\|_{\ctc{-\beta}}^{2 \hat\alpha - 1}\right)
+
(A/2 + \lambda) \expectation \left[ \int_0^t |Y^N_s - Y_s| ds \right]
+
B/2 \left\|b^N - b\right\|_{\ctc{-\beta}}^{2 \hat\alpha - 1}.
\end{equation*}
Finally, using Gronwall's lemma we get the following bound
    \begin{equation*}
        \mathbbm{E}\left[|Y^N_t - Y_t|\right]
        \leq
        B/2 \| b^N - b \|^{2 \hat\alpha - 1}_{\ctc{-\beta}} e^{(A/2+\lambda)t}
        + o\left(\left\|b^N - b\right\|_{\ctc{-\beta}}^{2 \hat\alpha - 1}\right).
    \end{equation*}
    We take expectation of both sides of \eqref{eq:|XN-X|}, {take the supremum over $t \in [0, T]$} and insert the above bound to conclude.
\end{proof}

\section{Numerical Implementation}\label{sec:numerics}

In this section we describe an implementation of our numerical scheme and analyse the results obtained.
Our implementation was done in the Python programming language and can be found in \cite{chaparrojaquezImplementationNumericalMethods2023}.
We recall that numerical implementation proceeds in two steps. First we approximate the drift $b$ with
$b^N$ as in \eqref{eqn:X^N}. 
{Then we apply the classical Euler-Maruyama scheme} \eqref{eq:m_em_approx} to approximate the SDE with drift $b^N$.
For the numerical example we will consider a time homogeneous drift, i.e.
$b \in \mathcal C^{(-\hat\beta)+}$.

Since the drift $b$ is a Schwartz distribution in $\mathcal C^{(-\hat \beta)+}$,  producing a numerical approximation  of it poses some challenges. We cannot simply discretise $b$ and then convolve it with the heat kernel to get $b^N$ as in \eqref{eq:bN}, see also \eqref{eq:heat_semigroup}, because the discretization of a distribution is not meaningful. 
Instead, we use the fact that an element of 
$\mathcal C^{(-\hat\beta)+}$
can be obtained as the distributional derivative of a function 
$h$
in 
$\mathcal C^{(-\hat\beta+1)+}$,
and that the derivative commutes with the heat kernel as explained in 
\eqref{eq:comm_sem_der}.
Indeed, since 
$b \in \hz{-\hat\beta+\epsilon}$
for some 
$\epsilon > 0$
and 
$-\hat\beta\in (-1/2,0)$,
we have 
$\gamma:=-\hat\beta+ \varepsilon+1>0$
and 
$h\in \hz{\gamma}$
is a function.
Using 
\eqref{eq:comm_sem_der}
allows us to discretise the function
$h$
first, then convolve it with the heat kernel, which has a smoothing effect, and finally take the derivative. 

Without loss of generality we will pick a function 
$h$
that 
is constant outside of a compact set,
which implies that the distribution 
$h'$
is supported on the same compact. For example one can choose 
$L>0$
large enough  {so that the solution 
$(X_t)$}
of the SDE with the drift 
$b\mathbbm 1_{[-L, L]}$
replacing 
$b$
will, with a large probability, stay within the interval 
$[-L, L]$
for all times 
$t\in[0,T]$,
rendering 
numerically
irrelevant the fact that the drift 
$b$
has been cut outside the compact support 
$[-L, L]$.

The simplest example of a drift 
$b \in \hz{\gamma-1}$,
$\gamma \in (1/2,1)$,
would be given by the derivative of 
the locally $\gamma$-H\"older continuous function $|x|^\gamma$, smoothly cut outside the compact set $[-L, L]$.
However, this function is not differentiable in 
$x=0$
only.
Instead,
for our numerical implementation we consider 
$h\in \mathcal C^\gamma$,
which is capable of fully encompassing the rough nature of the drift 
$b=h'$.
This can be obtained if the function 
$h$
has a `rough' behaviour in (almost) each point of the interval $[-L,L]$.
In this section, we take as $h$
a transformation of a trajectory
of a fractional Brownian motion 
(fBm)\footnote{A fractional Brownian motion 
$(B^H_x)_{x\geq 0}$
with Hurst parameter
$H\in(0,1)$
on a probability space 
$(\Omega, \mathcal F, \mathbb P)$
is a centered Gaussian stochastic process  with covariance given by 
$\mathbb E (B^H_x B^H_y) = \frac12 (|x|^{2H} + |y|^{2H} - |x-y|^{2H})$.
When 
$H=\tfrac12$
we recover a Brownian motion.}  
$(B^H_x)_{x\geq 0}$.
Indeed, it is known that 
{$\mathbb P$-almost 
all paths of a fBm}
$B^H$
are locally 
\holder{\gamma} continuous for any 
$\gamma < H$, 
but paths are almost nowhere differentiable, see 
\cite[Section 1.7]{biaginiStochasticCalculusFractional2008}.
We construct a compactly supported function 
$h$ as follows.
We take a path 
$(B_x^H (\omega))_{ x \in [0, 2L]}$
with Hurst index 
$H =  - \hat\beta+1 +2 \varepsilon \in(1/2,1)$
for some small 
$\varepsilon>0$,
which is thus locally 
$\gamma$-H\"older
continuous with 
$\gamma = - \hat\beta+1 + \varepsilon$.
The constraint 
$H\in(1/2,1)$
ensures that 
$\hat \beta \in (0,1/2)$,
as needed in Assumption 
\ref{ass:b_betahat}.
We define function $g:\R \to \R$ as
\begin{equation*}
    g(x) 
    =
    \big(B^H_x - \frac{B^H_{2L}}{2L} x\big)
    \mathbbm{1}_{\{x \in [0,2L]\}}
    (x),
\end{equation*}
which ensures that 
$g(0) = g(2L) = 0$.
This transformation, inspired by the so called Brownian bridge, ensures that 
$g$
is continuous and keeps the {same regularity as the fBm}
$B^H$.
Finally, for convenience, we translate 
$g$
so that it is supported in 
$[-L,L]$
rather than 
$[0,2L]$.
This is done for the sake of symmetry, since we choose to start our SDE from an initial condition close to 
$0$.
We choose 
$L$
large enough so that {the paths of $(X_t)$}
stay in the strip 
$[-L, L]$
with high probability up to our terminal time 
$T$. 
In summary, the function  
$h$
is constructed as 
\begin{equation}
    h(x) 
    = 
    \Big( B^H_{x+L}(\omega) - 
    \frac{B^H_{2L}(\omega)}{2L} (x-L) \Big)
    \mathbbm 1_{\{x \in [-L,L]\}}
    (x),
    \label{eq:fBbridge}
\end{equation}
and we  have 
$h \in \mathcal C^{\gamma} $
with 
$\gamma =-\hat \beta+ \varepsilon+1$,
so that 
$b  := h' \in \mathcal C^{(- \hat \beta)+ }$,
and both are supported on 
$[-L, L]$. 
By a slight abuse of notation we denote 
$B^H(x) =h(x) $
so that 
$b = (B^H)' $.

To compute 
$b^N$
we apply 
the semigroup 
$P_{1/N}$
to 
$b$,
as in 
\eqref{eq:bN},
which is equivalent to a convolution with the heat kernel 
$p_{1/N}$.
Since the derivative commutes with the convolution as recalled in \eqref{eq:comm_sem_der} we get
\begin{equation}
    \label{eq:numerical_integration}
    b^N(x)
    :=
    (P_{1/N} b)(x)
    =
    \left(p_{1/N}  \ast (B^H)'\right)(x)
    =
    \left(p_{1/N}' \ast B^H\right) (x)
    =
    \int_{-\infty}^{\infty}
    p_{1/N}' (y) B^H(x - y) dy.
\end{equation}
As the derivative of the heat kernel is
$p_{1/N}' (y) = - \frac{y}{1/N} p_{1/N} (y)$,
we have
\begin{equation}
    \label{eq:rep_der_hk}
    b^N(x) 
    = 
    - \int_{-\infty}^{\infty}
    B^H(x-y) 
    \frac y {1/N}
    p_{1/N} (y) dy
    .
\end{equation}

In order to approximate the above integral numerically, we create a uniform discretisation of the interval
$[-L, L]$ with $2M+1$ points,
denoted by $\suppDisc = \{ x_{-M}, ..., x_{M} \}$, with the mesh size $\delta: = \frac L M$.
We sample the fBm $B^H(x)$ in $x \in \suppDisc$ and {extend it to the real line} as a piecewise constant function $\hat B^H$ as follows:
\begin{equation}
    \label{eq:disc_fbm}
    \hat{B}^H(z)
    =
    \sum_{j=-M}^{M-1}
    B^H(x_j)
    \mathbbm{1}_{\{z \in [x_j-\frac{\delta}{2}, x_{j}+\frac{\delta}{2})\}} {(z)}, \qquad z \in \R.
\end{equation}
Inserting \eqref{eq:disc_fbm} into \eqref{eq:rep_der_hk} and performing the change of variable $z = x_i - y$ we obtain a numerical approximation of $b^N$ at any point $x_i \in \suppDisc$:
\begin{equation}
    \label{eq:approx_drift_3}
    \begin{split}
        b^N(x_i)
        &\approx
        - \int_{-\infty}^{\infty} \hat B^H(x_i-y) \frac{y}{1/N} p_{1/N} (y) dy\\
        &= 
        -
        \int_{-\infty}^{\infty}
        \sum_{j=-M}^{M}
        B^H(x_j)
        \mathbbm{1}_{\{z \in [x_j-\frac{\delta}{2}, x_{j}+\frac{\delta}{2})\}}
        \frac{x_i - z}{1/N}
        p_{1/N} (x_i - z)
        dz
        \\
        &=
        -
        \sum_{j=-M}^{M}
        B^H(x_j)
        \int_{x_j-\delta/2}^{x_{j}+\delta/2}
        \frac{x_i - z}{1/N}
        p_{1/N} (x_i - z)
        dz
        \\
        &=
        -
        \sum_{j=-M}^{M}
        B^H(x_j)
        \int_{-\delta/2}^{\delta/2}
        \frac{x_i - x_j - z'}{1/N}
        p_{1/N} (x_i - x_j - z')
        dz',
    \end{split}
\end{equation}
where we performed a second change of variable $z' = z - x_j$ in the last equality. Let us denote the integral in the last line of
\eqref{eq:approx_drift_3}
as
\begin{equation}
    \mathcal{I}^N(y)
    :=
    \int_{-\delta/2}^{\delta/2}
    \frac{y - z'}{1/N}
    p_{1/N} (y - z')
    dz',
    \label{eq:integral_delta_half}
\end{equation}
so that equation
\eqref{eq:approx_drift_3}
reads
\begin{equation}
    \label{eq:discrete_convolution}
    b^N(x_i)
    \approx
    -
    \sum_{j=-M}^{M}
    B^H(x_j)
    \mathcal I^N (x_i - x_j) 
    =: - ( B^H \ast \mathcal I^N) (x_i) ,
\end{equation}
where $\ast$ denotes the discrete convolution between vectors $(B^H(x_i))_{i=-M}^M$ and $(\mathcal{I}^N(x_i))_{i=-2M}^{2M}$. The latter vector 
needs to be formally defined for a $\delta$-discretisation of the interval $[-2L, 2L]$, 
however, in practice, the values of $\mathcal{I}^N(y)$ decrease quickly to $0$ as $|y|$ increases (this is illustrated in Figure \ref{fig:df} for two values of $1/N$), so it is enough to use the values of $(\mathcal{I}^N(x_i))$ for $x_i \in \suppDisc$.
We will use the above approximation \eqref{eq:discrete_convolution} in our numerical computations.

\begin{figure}[t]
    \centering
    \includegraphics[width=0.95\textwidth]{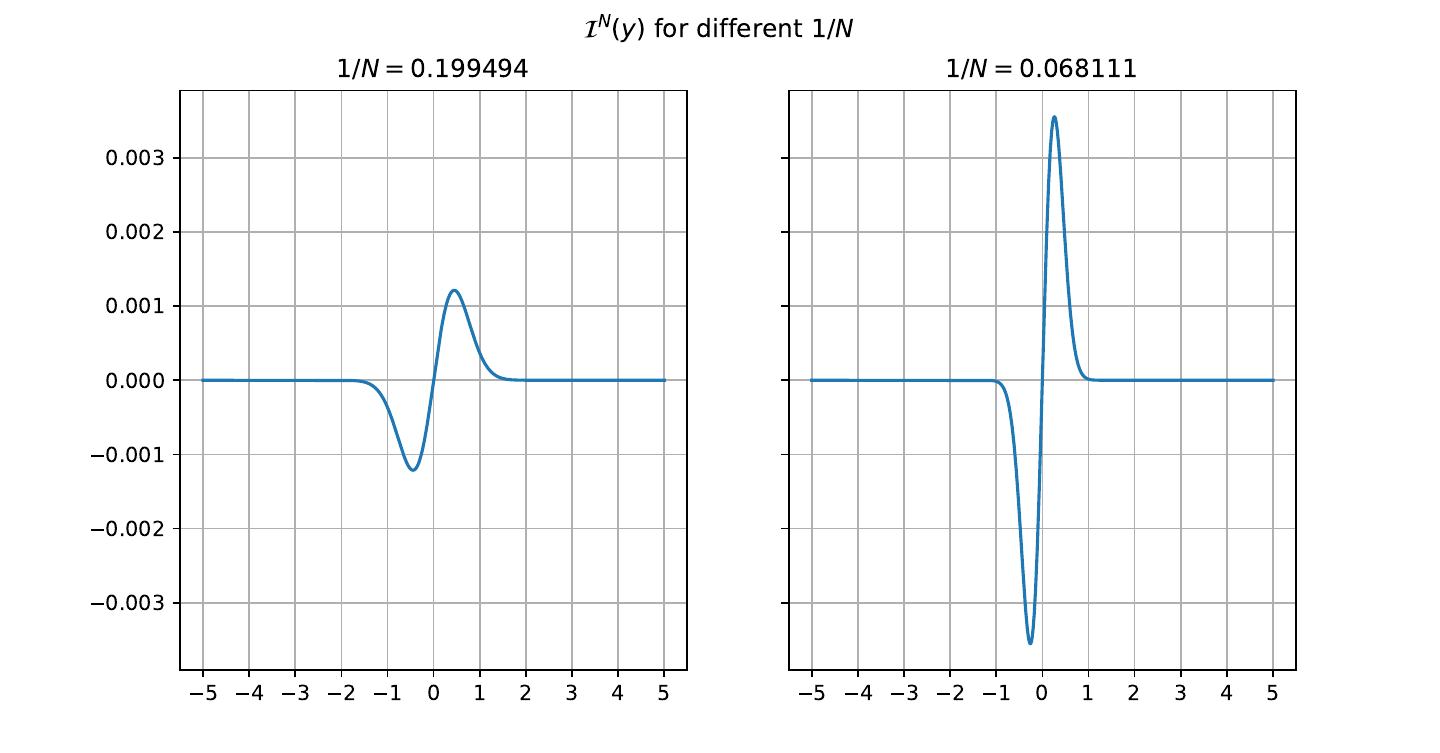}
    \caption{From left to right:
    $\mathcal I^N (y)$
    for
    $y \in [-5, 5]$
    and
    $1/N 
    =0.199494, 0.068111$
    respectively. }
    \label{fig:df}
\end{figure}

At this stage a few remarks regarding the size of $\suppDisc$ are in order. 
The drift $b^N$ is implemented with a discrete convolution given in equation \eqref{eq:discrete_convolution}, where one of the convoluting terms is $\mathcal I^N$, which is depicted in 
Figure
\ref{fig:df}. The peaks of $\mathcal I^N$  depend on the magnitude of $1/N$; indeed they become higher and thinner as  the variance $1/N$ of the heat kernel
$p_{1/N}$ decreases.  We choose $N= N(m) = m^{\eta}$ for  $\eta>0$ given in equation \eqref{eq:eta}, so  as the number 
of steps $m$ in the Euler scheme 
increases, the variance  $1/N(m) = m^{-\eta}$ decreases, leading to a vector $\mathcal I^{N(m)}$ of mostly zeros up to numerical precision. If the discretization $\suppDisc$ is too coarse compared to $1/N(m)$, the vector $\mathcal I^{N(m)}$ does not retain enough information. To avoid this issue we have to make sure that there are sufficiently many points in $\suppDisc$, i.e.\ that $2M$ is large enough compared to the size of the support $[-L, L]$. 
More specifically, we require that the distance $\delta$  between points in the grid is less than the standard deviation of the heat kernel $\sqrt{1/N(m)}$, i.e. $\delta < \sqrt{1/N(m)} = m^{-\eta/2}$ for every $m$ considered in numerical computations, which leads to 
\begin{equation}
	\label{eq:lb_A}
	2M>2Lm^{\eta/2}.
\end{equation}
As an example, let us consider a discretisation of the interval
$[-5, 5]$ (hence $L=5$)
and use
$m =2^{12}$
points in the Euler-Maruyama scheme. We will need at least $2M+1> 2Lm^{\eta/2}+1  =  10 \cdot  2^{6\eta}+1 $ points in 
$\suppDisc$, where the value $\eta$ depends on $\hat\beta$ and it is given explicitly in equation  \eqref{eq:eta}.
The smallest admissible number of points $2M+1$ of the discretisation $\suppDisc$ for varying $\hat\beta$ is given in Table \ref{table:beta}.
\begin{table}[t]
    \centering
    \begin{tabular}{|c||c|c|c|c|c|}
        \hline
        $\hat \beta$ & $10^{-6}$ & $1/16$ & $1/8$ & $1/4$ & $1/2$\\
        \hline
        $1/N(m) = m^{-\eta(\hat \beta)}$ & 0.001288 & 0.001398 & 0.001489 & 0.001768 & 0.003906\\
        $2M+1$ & 279 & 269 & 261 & 239 & 161\\
        \hline
    \end{tabular}\\[5pt]
    
    \caption{Minimum amount of points $2M+1$ needed in the discretisation $\suppDisc$ when $m=2^{12}$ and $L=5$. This varies according to the variance $1/N$, which in turn is a function of $\hat \beta$.}
    \label{table:beta}
\end{table}

We now summarise the procedure to compute numerically an approximation of the drift $b^{N(m)}$.
\begin{enumerate}
    \item Fix {$\hat \beta$, $L$, and the number $m$} of steps of the Euler scheme. Compute the smallest integer $M$ that satisfies \eqref{eq:lb_A}. Define the discretisation $\suppDisc$ and the corresponding mesh size $\delta$. 
    \item 
    Simulate a single path of fBm on the interval $[0,2L]$ with mesh size $\delta$ and 
    with a given Hurst parameter
    $H = -\hat \beta+1+\epsilon$ for some small $\epsilon > 0$.
    \item
    Transform the path of fBm into a bridge on $[-L, L]$ by applying the transformation \eqref{eq:fBbridge} to get a vector $B^H(x_i)$ for $x_i \in \suppDisc$. 
    \item
    Compute numerically the integral
    $\mathcal I^{N(m)} (x_i)$ for $x_i \in \suppDisc$.
    \item
    Perform the discrete convolution  
    $-(B^H \ast \mathcal I^{N(m)})$ as in \eqref{eq:discrete_convolution} to approximate numerically $b^{N(m)}(x_i)$ for $x_i \in \suppDisc$.
    \item
    Extend $b^{N(m)}(x)$ for all $x\in[-L, L]$ by linear interpolation.
\end{enumerate}

Once we have a numerical approximation of the drift coefficient, the remaining step is to apply the standard Euler-Maruyama scheme. To calculate the empirical rate of convergence of the numerical scheme we must have approximations with an increasing number of steps $m$ as well as  a proxy of the real solution, since the real solution is unknown in a closed form. 
The strong error of the scheme is calculated by Monte Carlo approximation of the $L^1$ norm of the difference between those approximations and the proxy at time $T$.
The procedure reads as follows:
\begin{enumerate}
    \item Choose a `large' $m$ for the proxy of the real solution, and $m_i << m$, $i=0, \ldots, I$, for the approximations and such that each $m_i$ divides $m$. Choose a number of sample paths $Q\in\mathbb N$ sufficiently large.
    \item Denote by $\Delta t_i= T/m_i$ the time-step for the approximated solutions, where $T$ is the terminal time.
    \item  Define $\suppDisc$ with $M$ points, where $M$ and $m$ satisfy  \eqref{eq:lb_A}.
    Generate one fBm path on the discrete grid $\suppDisc$.
    \item Run the Euler-Maruyama scheme for the proxy solution and for the approximated solutions up to time $T$, with the same $Q$ Brownian motion paths.
    \item 
    Compute the strong error between the proxy solution corresponding to $m$ and the approximations corresponding to $m_i$, by calculating a Monte Carlo average of the absolute differences between computed solutions at time $T$ across the $Q$ sample paths. Denote this approximations of the strong error by $\epsilon_i$, $i=0, \ldots, I$.
    \item As we expect $\epsilon_i \approx c {m_i}^{-r} = c (\Delta t_i)^r$, where $r$ is the convergence rate,  we compute the empirical rate $r$  by performing a linear regression of $\log_{10}(\epsilon_i)$ on $\log_{10}(\Delta t_i)$, $i=0, \ldots, I$.
\end{enumerate}

\begin{figure}[t]
    \centering
    \includegraphics[width=\textwidth]{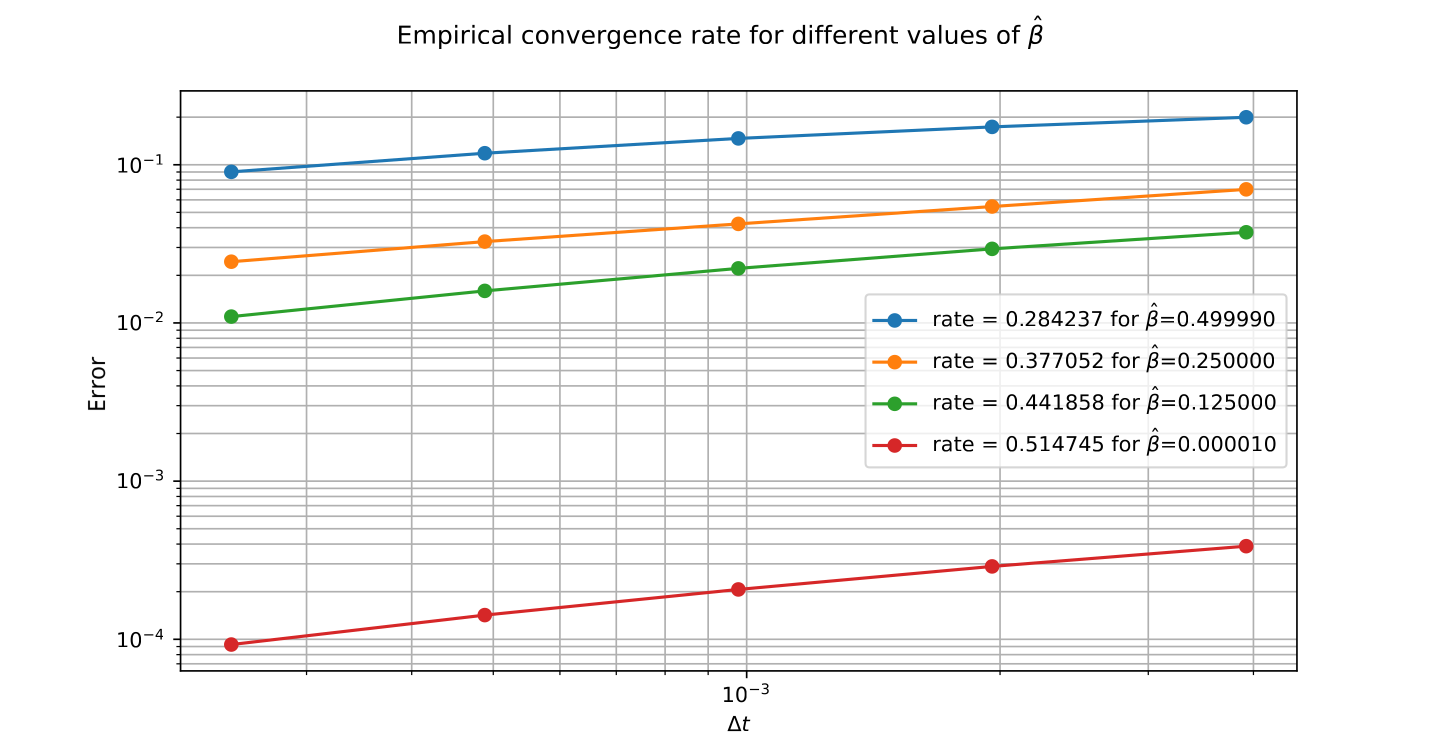}
    \caption{Examples of empirical  convergence rates for
    $\hat \beta = \varepsilon, 1/8, 1/4, 1/2-\varepsilon$, with $ \varepsilon=10^{-6}$, 
    obtained running an Euler scheme with 
    $Q=10^4$
    sample paths,  
    $T=1$,
    $m=2^{15}$
    points for the proxy of the real solution and 
    $m_i = 2^{8+i}$
    for 
    $i=0, \ldots, 4$,
    for the approximated solutions.
    The empirical convergence rate $r$ for each $\hat \beta$ is the slope of its corresponding line, which is plotted in a log-log graph.
   }
    \label{fig:rates}
\end{figure}

In 
Figure \ref{fig:rates}
we plot the empirical convergence rate we obtained for different choices of the smoothness parameter $\hat \beta$, that is for 
$\hat \beta = \varepsilon, 1/8, 1/4, 1/2-\varepsilon$, with $ \varepsilon=10^{-6}$, and with the rest of the parameters as indicated in the caption to Figure \ref{fig:rates}. 
 Note that as    $\hat \beta$   grows and the drift becomes more rough, 
   the empirical convergence rate  becomes smaller and at the same time the strong error increases. 
Note that the empirical  convergence rate is  close to
$1/2$
when
$\hat \beta \approx 0$,
which agrees with the theoretical results obtained by \cite{dareiotisQuantifyingConvergenceTheorem2021} in the realm of measurable functions, see also \cite{butkovskyApproximationSDEsStochastic2021}. Indeed, they show a strong  convergence  rate of $1/2+\alpha/2$ when $b\in C^{\alpha}$ for $\alpha\in[0,1)$, which reduces to $1/2$ for measurable functions ($\alpha=0$).
On the other hand,  for
$\beta \to 1/2$
we have  an empirical  convergence rate close to
$1/4$.
{We now refer to the comment made in Remark \ref{rm:8}. The nearly linear dependence of the log-error on $\log(\Delta t)$ suggests that the bound in Proposition \ref{prop:reg_to_original} holds. We are not able to directly verify the condition $\|b^N - b\|_{C_T \mathcal C^{-\beta}} < 1$ because the drift $b$ is not available.}

Finally, we performed a further experiment to better compare the empirical rate with the theoretical results. Since the drift of the SDE is obtained running a single path of a fBm, and clearly there is randomness there, we decided to run the algorithm for 50 different paths, for each value $\hat \beta$, and then we computed the average of the empirical convergence rates as well as its 95\% confidence interval. We compared this with the theoretical rate obtained in Theorem \ref{th:convergence_rate_es} and with the conjecture that the rate should be $1/2 - \hat \beta/2$. The latter  would be the natural extension  of the results of \cite{dareiotisQuantifyingConvergenceTheorem2021} if they could be extended into the case of distributions, in particular  with $-\hat \beta \in (-1/2,0)$ which is the case we treat here.  We collected the results in Table \ref{table:rates} below and also plotted them in Figure \ref{fig:empirical_rate}. This experiment strongly suggests that  our theoretical result is not optimal, and that the convergence rate indeed could be $1/2 - \hat \beta/2$. Further studies are needed to prove or disprove this conjecture. 
\begin{table}[t]
    \centering
    \begin{tabular}{|c|c|c|c|c|c|c|c|}
        \hline
        $\hat \beta$ & $\varepsilon = 10^{-6}$ & $1/16$ & $1/8$ & $1/4$ & $3/8$ &  $7/16$ &  $1/2 - \varepsilon$\\
        \hline
         Average of empirical  rates  & 0.50814  & 0.48542   &   0.44063  & 0.36960    & 0.30478  &  0.29048 & 0.28275   \\
        $1/2 - \hat \beta/2$    &  0.49999     &  0.46875  &  0.43750   & 0.37500  &  0.31250 &0.28125   & 0.25000 \\
        Theoretical rate    &  0.16666  & 0.13243   & 0.10000  &   0.04545 & 0.01111   & 0.00270    &  0.00000   \\
        \hline
    \end{tabular}\\[5pt]
    
    \caption{Average of empirical convergence rates obtained using 50 Euler-Maruyama approximations with $10^4$ sample paths for each $\hat \beta$, the conjecture rate of $1/2-\hat \beta/2$ and the theoretical rate from Theorem \ref{th:convergence_rate_es}. All values rounded to 5 decimal places.
    \label{table:rates}
    Same values are plotted  in Figure \ref{fig:empirical_rate}.}
\end{table}

\begin{figure}[t]
    \centering
    \includegraphics[width=\textwidth]{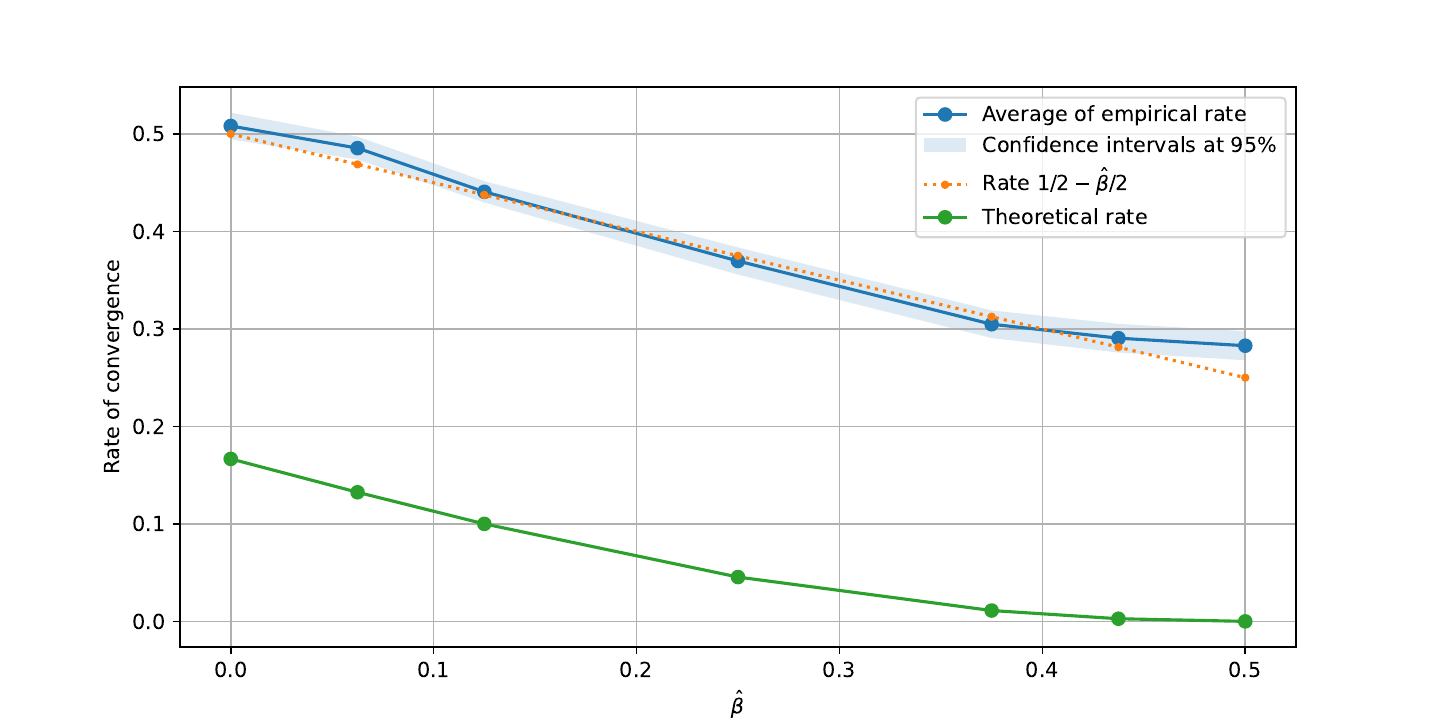}
    \caption{Plot of average of empirical convergence rates obtained as average of 50 Euler-Maruyama approximations with $10^4$ sample paths for each $\hat \beta$. The plot contains also the 95\% confidence interval for the average,  the conjectured rate  of $1/2 - \hat \beta/2$  and the theoretical rate obtained in Theorem \ref{th:convergence_rate_es}.
  }
    \label{fig:empirical_rate}
\end{figure}

\vspace{10pt}
{\noindent\bf Data Availability Statement:} The research code described in Section \ref{sec:numerics}  of this article is available on Zenodo.
See the repository ``Implementation of the Numerical Methods from {`Convergence rate of numerical solutions to SDEs with distributional drifts in Besov spaces'}'', with DOI \texttt{10.5281/zenodo.8239428}, and with identifier \cite{chaparrojaquezImplementationNumericalMethods2023} in the references list.

\vspace{10pt}
{\noindent\bf Acknowledgement:} The author E.\ Issoglio acknowledges partial financial support under the National Recovery and Resilience Plan (NRRP), Mission 4, Component 2, Investment 1.1, Call for tender No. 104 published on 2.2.2022 by the Italian Ministry of University and Research (MUR), funded by the European Union – NextGenerationEU– Project Title “Non–Markovian Dynamics and Non-local Equations” – 202277N5H9 - CUP: D53D23005670006 - Grant Assignment Decree No. 973 adopted on June 30, 2023, by the Italian Ministry of University and Research (MUR).

\printbibliography

\end{document}